\documentclass[11pt]{amsart}

\usepackage{graphicx,amsthm}
\usepackage{enumerate}
\usepackage{empheq}
\usepackage{amssymb}
\usepackage{array}
\usepackage{tikz}
\usepackage{tikz-cd}
\usepackage{amsmath,amscd,amsfonts,bbm, mathabx}
\allowdisplaybreaks[4]
\usepackage{ragged2e}
\usepackage{a4wide}
\usepackage{hyperref}
\usepackage{xcolor}

\usepackage{diagbox}
\usepackage{url}

\usepackage[mathscr]{euscript} % added by lee
\usepackage{booktabs} % added by lee
\usetikzlibrary{
	3d, matrix,decorations.pathreplacing,calc,decorations.pathmorphing}
\usetikzlibrary{decorations.markings}
\usetikzlibrary{cd,arrows}
\usetikzlibrary{arrows.meta}
\usetikzlibrary {positioning}
\usetikzlibrary{backgrounds,scopes}

\usepackage{colonequals}

\usepackage[normalem]{ulem} %strike out

\usepackage{caption}
\usepackage{subcaption}
\newtheorem{theorem}{Theorem}[section]
\newtheorem{lemma}[theorem]{Lemma}
\newtheorem{proposition}[theorem]{Proposition}
\newtheorem{corollary}[theorem]{Corollary}

\newtheorem{question}[theorem]{Question}

\theoremstyle{definition}
\newtheorem{definition}[theorem]{Definition}
\newtheorem{example}[theorem]{Example}

\newtheorem{remark}[theorem]{Remark}

\numberwithin{equation}{section}

\newtheorem{thmy}{Theorem}
\newcommand{\C}{\mathbb{C}}
\newcommand{\Cstar}{\mathbb{C}^{\ast}}
\newcommand{\flag}{F\ell}
\renewcommand{\S}{\mathfrak{S}}

\DeclareMathOperator{\GL}{GL}

\DeclareMathOperator{\Hess}{Hess}

\newcommand{\Gwh}[1]{\Gamma(\Ow{{#1}} \cap \Hess(s,h))}
\newcommand{\Owh}[1]{\Omega_{{#1},h}}
\newcommand{\Ow}[1]{\Omega_{{#1}}}
\newcommand{\Owo}[1]{\Omega_{{#1}}^{\circ}}
\newcommand{\OwoY}[1]{\Omega_{{#1},Y}^{\circ}}
\newcommand{\OwY}[1]{\Omega_{{#1},Y}}

\newcommand{\TE}{T\mkern-2mu E}

\newcommand{\hpat}[1]{{#1}_h}
\allowdisplaybreaks

\begin{document}
\title[Intersections of Schubert varieties and smooth $T$-stable subvarieties of $G/B$]{
Intersections of Schubert varieties and smooth $T$-stable subvarieties of flag varieties
}

\author{Jaehyun Hong}
\address{Center for Complex Geometry, Institute for Basic Science (IBS), Daejeon 34126, Republic of Korea}
\email{jhhong00@ibs.re.kr}

\author{Eunjeong Lee}
\address{Department of Mathematics,
	Chungbuk National University,
	Cheongju 28644, Republic of Korea}
\email{eunjeong.lee@chungbuk.ac.kr}

\author{Seonjeong Park}
\address{Department of Mathematics Education, Jeonju University, Jeonju 55069, Republic of Korea}
\email{seonjeongpark@jj.ac.kr}

\date{\today}

\keywords{Hessenberg varieties, Schubert varieties, Bia{\l}ynicki-Birula decompositions, GKM varieties}

\subjclass[2020]{Primary 14M15; Secondary 05E14, 14L30}

%%%%%%%%%%
\begin{abstract}
A smooth projective variety equipped with a suitable torus action admits a cell decomposition, called the Bia{\l}ynicki–Birula decomposition. The singularities of the closures of these cells are not well understood in general. Schubert varieties in a flag variety~\(G/B\) are classical examples of such closures, and several criteria for their smoothness are known. In this paper, we study the closures of Bia{\l}ynicki-Birula cells in regular semisimple Hessenberg varieties~\(\operatorname{Hess}(s,h)\), called Hessenberg Schubert varieties. We first investigate the smoothness of intersections of Schubert varieties with \(\operatorname{Hess}(s,h)\), obtaining a sufficient condition for a Hessenberg Schubert variety to be smooth. We also establish several combinatorial sufficient conditions for the smoothness of Hessenberg Schubert varieties in type \(A\).
\end{abstract}

\thanks{Hong was supported by the Institute for Basic Science IBS-R032-D1.
Lee was supported by the National Research Foundation of Korea(NRF) grant funded by the Korea government(MSIT) (No.\ RS-2023-00239947 and RS-2026-25490065) and the POSCO Science Fellowship of the POSCO TJ Park Foundation. Park was supported
 by the National Research Foundation of Korea(NRF) grant funded
by the Korea government(MSIT) (Nos.\ NRF-2020R1A2C1A01011045 and RS-2026-25490927).}

\maketitle

\setcounter{tocdepth}{1}
\tableofcontents

%%% section 1
\section{Introduction}\label{sec:Intro}
The \emph{Schubert varieties} are subvarieties of the flag variety $\flag(\C^n) $ obtained by taking the closure of cells in the Bruhat decomposition of $\flag(\C^n)$.
The cohomology classes of Schubert varieties form an additive basis of the cohomology ring of the flag variety, and understanding their geometry plays an important role in studying the geometry of the flag variety. Moreover, there are interesting avenues connecting the geometry of Schubert varieties, representation theory, and combinatorics, as exhibited by the theory of Schubert calculus.

Not every Schubert variety is smooth, but the singularities of Schubert varieties have been widely studied. To be more precise, the Schubert varieties in $\flag(\C^n)$ are parametrized by the symmetric group $\S_n$ on $[n] \colonequals \{1,\dots,n\}$, and we denote by $\Omega_w$ the (opposite) Schubert variety indexed by a permutation $w \in \S_n$.\footnote{In the literature, both Schubert varieties and opposite Schubert varieties are considered. In this paper, we focus on the opposite Schubert varieties. For brevity, we will refer to opposite Schubert varieties simply as Schubert varieties, when no confusion arises. (See Subsection~\ref{subsection_flag} for the definition of Schubert varieties.)} There are several well-known criteria for determining the smoothness of $\Omega_w$: one uses the \emph{GKM subgraph} of $\Omega_w$ in that of $\flag(\C^n)$ (cf.~\cite{Carrell94}), and another uses the pattern avoidance of $w$ (cf.~\cite{Billey98}).

In this paper, we extend our attention to the family of \emph{regular semisimple Hessenberg varieties}, each of which is a smooth subvariety of $\flag(\C^n)$; the flag variety $\flag(\C^n)$ itself belongs to this family.
To define a Hessenberg variety, we need two pieces of data: one is a \emph{Hessenberg function} $h \colon [n] \to [n]$; the other is a linear operator $x \colon \C^n \to \C^n$. A function $h \colon [n] \to [n]$ is called a Hessenberg function if $h(i) \geq i$ for all $i \in [n]$ and $h(i) \leq h(i+1)$ for $i \in [n-1]$. The \emph{Hessenberg variety} $\Hess(x,h)$ is defined by the set of all flags $V_{\bullet} \in \flag(\C^n)$ satisfying $x V_i \subseteq V_{h(i)}$ for all $i$. When $x = s$ is regular semisimple, we call $\Hess(s, h)$ a \emph{regular semisimple Hessenberg variety}.

A regular semisimple Hessenberg variety $\Hess(s, h)$ is stable under the action of the torus~$T$, the centralizer of $s$ in $G = \GL_n(\C)$, which is a maximal torus since $s$ is regular semisimple. 
Notice that the $T$-fixed points in $\Hess(s, h)$ are parametrized by $\S_n$.
With respect to a generic one-parameter subgroup of $T$, the Bia{\l}ynicki-Birula decomposition of $\Hess(s,h)$ has cells given by the intersections of $\Hess(s,h)$ with Schubert cells $\Omega_w^{\circ}$. 
We define the \emph{Hessenberg Schubert variety} $\Owh{w}$ to be the closure of the intersection $\Owo{w} \cap \Hess(s,h)$.

 \medskip

The main goal of this paper is to study the smoothness of a Hessenberg Schubert variety~$\Owh{w}$, and that of the intersection $\Ow{w} \cap \Hess(s,h)$ in terms of the graph~$\Gamma(\Ow{w} \cap \Hess(s,h))$, which is a certain induced subgraph of the GKM graph of $\Hess(s,h)$. 

For this purpose, we consider a more general setting. 
Let $X$ be a smooth GKM variety equipped with an algebraic
action of a torus~$T$. 
For a   $T$-stable subvariety $V\subseteq X$ and $v\in V^T$, let
$E(V,v)$ denote the set of $T$-stable curves in $V$ containing $v$,
and set
\[
 \TE(V,v)
:=
\sum_{C\in E(V,v)}T_vC
\subseteq T_vV.
\]
Note that if $V$ is smooth, then the equality holds. 
Indeed, since $X$ is GKM and $Y$ is smooth and $T$-stable, every weight line
in $T_vY$ is tangent to a unique   $T$-stable curve
in $Y$.  Hence, we get 
\( T_vY=\TE(Y,v)\).
Also, if $V$
is a Schubert variety in~$G/B$ in type~$A$, then the equality holds (see~\cite[Theorem~1]{LS84}). However, in general, the inclusion is proper (see~\cite{GrahamKreiman}). 

\begin{theorem}\label{thm_main}
Let $X$ be a smooth GKM variety equipped with an  
action of an algebraic torus~$T$  and let
\(
X=\bigsqcup_{w\in\mathcal S}\Omega_w^\circ
\)
be the Bia{\l}ynicki-Birula decomposition with respect to a generic one-parameter subgroup of~$T$. Set
$\Omega_w=\overline{\Omega_w^\circ}$, and let $Y\subseteq X$ be a
smooth  $T$-stable subvariety.  Suppose that  \(\Omega_w^\circ\cap Y\neq\emptyset\),
$\Gamma(\Omega_w \cap Y)$ is connected, and
\[
T_v\Omega_w=\TE(\Omega_w,v)
\qquad
\text{for every }v\in (\Omega_w \cap Y)^T.
\]
Then the following statements are equivalent. 
\begin{enumerate}
    \item  $\Gamma(\Omega_w \cap Y)$ is $d$-regular, where $d = \dim(\Owo{w} \cap Y)$
    \item $\Omega_w \cap Y$ is smooth.
\end{enumerate}
In particular, if one of the above conditions holds, then 
\(\Omega_w \cap Y
=
\overline{\Omega_w^\circ\cap Y}\).
\end{theorem}

Theorem~\ref{thm_main} can be applied to flag varieties in type $A$.
One direction $\eqref{statement_smooth} \implies \eqref{statement_regular}$ is clear: It follows from the fact that $\Ow{w} \cap Y$ has the induced $T$-action and it satisfies the GKM condition. Indeed, the graph $\Gamma(\Ow{w} \cap Y)$ encodes the data of $T$-fixed points and $T$-stable curves.
Therefore, the number of $T$-stable curves containing a fixed point $v$ is the same as the number of edges in $\Gamma(\Ow{w} \cap Y)$ incident to the vertex corresponding to $v$, which is again the same as the dimension of the tangent space at that point.
To obtain the converse statement, we compare $\TE(\Omega_w \cap Y,v)$ and $\TE(\Omega_w,v) \cap \TE(Y,v)$. 
See Section~\ref{sec:Irr} for more details.

\medskip 
Returning to the case when $X=F\ell(\mathbb C^n)$,   
since  the equality 
$T_v\Omega_w=\TE(\Omega_w,v)$
  holds 
  for every  $v\in  \Omega_w ^T$, 
we apply Theorem~\ref{thm_main} to the intersection of~$\Ow{w}$ with a Hessenberg variety $\Hess(s,h)$, to obtain the following sufficient condition for the smoothness of the Hessenberg Schubert variety~$\Owh{w}$.
We remark that $\Omega_w \cap \Hess(s,h)$ is connected if $\Hess(s,h)$ is connected (Proposition~\ref{prop:connected_graph}). 
\begin{theorem}\label{thm_main_Hess}
Let $s$ be a regular semisimple element and let $h$ be a Hessenberg function.
Suppose that $\Omega_w \cap \Hess(s,h)$ is connected.
Then the following statements are equivalent.
\begin{enumerate}
    \item $\Gamma(\Ow{w} \cap \Hess(s, h))$ is regular. \label{statement_regular}
    \item $\Ow{w} \cap \Hess(s, h)$ is smooth. \label{statement_smooth}
\end{enumerate}
In particular, if one of the above conditions holds, then $\Owh{w}$ is smooth, and we have $\Ow{w} \cap \Hess(s, h) = \Owh{w}$.
\end{theorem}

\medskip 
Because of the construction, a Hessenberg Schubert variety $\Owh{w}$ admits the action of $T$ and $\Owh{w}^T \subseteq (\Omega_w \cap \Hess(s,h))^T$.
A permutation~$w\in \S_n$ is \emph{$h$-admissible} if $\Omega_{w,h}^T = (\Omega_w \cap \Hess(s,h))^T$ (see Definition~\ref{def_admissible} and Proposition \ref{prop:fixed_points_admissible}).
 
As shown in~\cite[Theorem~A]{CHL2}, the set of $h$-admissible permutations corresponds to a set of Hessenberg Schubert cohomology classes which \emph{generates} the cohomology group $H^*(\Hess(s,h);\C)$ as an $\S_n$-module.

Theorem~\ref{thm_main_Hess} provides a \emph{combinatorial} criterion on the graph $\Gamma(\Ow{w} \cap \Hess(s,h))$ for the smoothness of the Hessenberg Schubert variety $\Owh{w}$.
However, in practice, the regularity of the graph $\Gamma(\Ow{w} \cap \Hess(s,h))$ is too strong to apply to arbitrary permutations.
Instead of applying Theorem~\ref{thm_main_Hess} directly, 
we first replace \(w\) by its unique \(h\)-admissible representative \(\widetilde{w}\), characterized by
\[
    \Gamma(\Owh{w}) \cong \Gamma(\Owh{\widetilde{w}}).
\]
Indeed, the corresponding Hessenberg Schubert varieties serve as \emph{prototypes} for other Hessenberg Schubert varieties. See Proposition~\ref{prop:prop_h-admissible} and Corollary \ref{cor. graph}.

We note that the regularity of the graph $\Gwh{w}$ for an $h$-admissible $w$ has been deeply studied in~\cite{CHP}.
For example, to show the regularity of the graph $\Gwh{w}$, it suffices to check the regularity only at $w_0$, the longest element in $\S_n$ (see Corollary~\ref{cor regularity}). Furthermore,
the graph $\Gwh{w}$ is regular if and only if the permutation $w$ avoids certain patterns (see Theorem~\ref{thm:pattern_avoidance}).

Using these descriptions together with the regularity of the graph $\Gamma(\Omega_{\widetilde{w}} \cap \Hess(s,h))$, we obtain the following sufficient conditions for
the smoothness of the Hessenberg Schubert variety $\Owh{w}$:
\begin{theorem} 
\label{thm_regular_implies_smoothness}
    For $w\in \S_n$, let $\widetilde{w}$ be the $h$-admissible permutation corresponding to $w$. Then the following statements are equivalent:
    \begin{enumerate}
        \item $\Gwh{\widetilde{w}}$ is regular.
        \item $\Ow{\widetilde{w}}\cap\Hess(s,h)$ is smooth.
        \item $\Ow{\widetilde{w}}\cap\Hess(s,h)$ is irreducible and its Poincar\'e polynomial %\(P_{\widetilde{w},h}(t)\) 
        is palindromic.
        \item $\widetilde{w}$ avoids all associated patterns:
        $$\hpat{2143},\,\hpat{1324},\,\hpat{1243},\,\hpat{2134},\,\hpat{1423},\,\hpat{2314},\,\hpat{2413}.\,$$
    \end{enumerate}In particular, if one of the above conditions holds, then $\Ow{w,h}$ is smooth.

\end{theorem}

One may wonder whether the converse of the last statement of Theorem~\ref{thm_regular_implies_smoothness} holds, but the converse does not hold in general, as is exhibited in Example~\ref{example_reducible}. That is because the graph $\Gwh{w}$ basically encodes the data of $\Ow{w} \cap \Hess(s,h)$, not that of $\Owh{w}$.

\medskip 
The organization of the paper is as follows. In Section~\ref{sec:Prelim}, we recall the basic notions of GKM varieties and some examples, including flag varieties and Hessenberg varieties. In Section~\ref{sec:Irr}, we prove Theorem~\ref{thm_main} after investigating the irreducibility of the intersection $\Ow{w} \cap Y$ and the smoothness under the regularity assumption. We prove Theorem~\ref{thm_regular_implies_smoothness} in Section~\ref{sec:Hess} by applying Theorem~\ref{thm_main} to the case of regular semisimple Hessenberg varieties in type $A$.  
We discuss further questions in Section~\ref{sec:questions}.

%%% section 2
\section{GKM varieties and their examples}\label{sec:Prelim}

In this section, we provide a brief overview of GKM theory following~\cite{GKM,Tymoczko}. We introduce flag varieties and Hessenberg varieties as examples of GKM varieties and describe the cohomology of a Hessenberg variety within this framework. We also review the definitions and key properties of Hessenberg Schubert varieties.

\subsection{GKM theory and GKM cohomology}
Let $X$ be a (possibly singular) complex projective algebraic variety with an action of the algebraic torus $T\cong (\C^\ast)^n$. We say that $X$ is a \emph{GKM \textup{(}Goresky--Kottwitz--MacPherson\textup{)} variety} if it satisfies the following three conditions:
\begin{itemize}
    \item $X$ has finitely many $T$-fixed points,
    \item there are finitely many (complex) one-dimensional orbits of $T$ on $X$, and
    \item $X$ is equivariantly formal with respect to the action of $T$.
\end{itemize}
The third property is rather technical but holds when $X$ has no odd-degree ordinary cohomology. More precisely, if $H^{odd}(X)$ vanishes, then $X$ is equivariantly formal with respect to every torus action.

After choosing the standard characters of $T$, we identify $H_T^{\ast}(\mathrm{pt}; \mathbb{C})$ with the polynomial ring $\mathbb{C}[t_1,\dots,t_n]$.
Based on the zero- and one-dimensional orbits of a GKM variety $X$, the \emph{GKM graph} $\Gamma=(V, E, \alpha)$ is a directed graph whose edges are labeled by $\alpha$, and it is constructed as follows:
\begin{enumerate}
    \item $V=X^T$, the set of $T$-fixed points of $X$;
    \item a directed edge $v\to w$ belongs to $E$ if and only if there exists a one-dimensional orbit $\mathcal{O}_{v,w}$ of $T$ on $X$ whose closure $\overline{\mathcal{O}_{v,w}}$ has $v$ and $w$ as its $T$-fixed points; and
    \item we label each directed edge $v\to w$ with the weight $\alpha(v\to w)\in\mathbb{Z}[t_1,\dots,t_n]$ of the $T$-action corresponding to the closure $\overline{\mathcal{O}_{v,w}}$ at the fixed point $v$.
\end{enumerate}
Note that condition (2) implies that $v\to w$ belongs to $E$ if and only if $w\to v$ belongs to $E$, and the label~$\alpha$ satisfies $\alpha(v\to w)=-\alpha(w\to v)$.

Using the GKM graph, Goresky, Kottwitz, and MacPherson~\cite{GKM} describe the equivariant cohomology ring of a GKM variety.

\begin{theorem}\cite[Theorem~1.6.2]{GKM}\label{thm_GKM_cohomology}
    Let $X$ be a GKM variety with the GKM graph $\Gamma=(V, E, \alpha)$. Then a map $p\colon X^T\to \mathbb{C}[t_1,\dots,t_n]$ represents a cohomology class in $H^\ast_T(X)$ if and only if it satisfies a compatibility condition:
    $$p(v)\equiv p(w)\pmod{\alpha(v\to w)}\text{ for all }v\to w \text{ in }E,$$ and
    the equivariant cohomology ring $H_T^\ast(X;\C)$ is isomorphic to
    \begin{equation*}
        \begin{split}
            &H^\ast(\Gamma;\mathbb{C})\\
            &\coloneqq \left\{(p(v))_{v \in V}\in \bigoplus_{v\in V}\mathbb{C}[t_1,\dots,t_n]\,\middle|\,\alpha(v\to w)\mid (p(v)-p(w))\text{ for all }v\to w \text{ in }E\right\}
        \end{split}
    \end{equation*} as a $\C[t_1,\dots,t_n]$-algebra.
\end{theorem}
We call an element of $H^\ast(\Gamma;\mathbb{C})$ a \emph{GKM cohomology class}. For a $T$-stable subvariety $Y$ of $X$, we denote by $[Y]_T$ the GKM cohomology class corresponding to $Y$ in $H_T^\ast(X)$.

It follows from Theorem~\ref{thm_GKM_cohomology} that the ordinary cohomology ring of a GKM variety $X$ is
$$H^\ast(X;\C)\cong H^\ast(\Gamma;\C)/(t_1,\dots,t_n),$$ where we use $t_i$ to indicate the element in $H^\ast(\Gamma;\C)$ whose value at each $v\in V$ is $t_i$.

%%%%
\subsection{Flag varieties}\label{subsection_flag}
Let $G$ be the general linear group $\GL_n(\mathbb{C})$ and let $B$ (resp. $B^-$) be the subgroup of upper (resp. lower) triangular matrices in $G$. Then the flag variety $\flag(\mathbb{C}^n)$ is the homogeneous space $G/B$, which is a smooth projective variety of dimension $d=\binom{n}{2}$. We can identify $\flag(\mathbb{C}^n)$ with
\[\{V_\bullet=( \{0\}\subsetneq V_1\subsetneq V_2\subsetneq\cdots\subsetneq V_n=\mathbb{C}^n)\mid \dim V_i=i\text{ for all }i=1,\dots,n\},\]
where each $V_\bullet$ is called a \emph{flag}. For example, each permutation $w\in\S_n$ defines a flag
\[(\{0\}\subsetneq \langle \mathbf{e}_{w(1)}\rangle\subsetneq \langle \mathbf{e}_{w(1)},\mathbf{e}_{w(2)}\rangle\subsetneq\cdots\subsetneq \langle \mathbf{e}_{w(1)},\dots,\mathbf{e}_{w(n)}\rangle=\C^n),\]
called a \emph{coordinate flag}, where $\mathbf{e}_1,\dots,\mathbf{e}_n$ are standard basis vectors. If there is no confusion, we will denote the coordinate flag defined by the permutation \( w \) simply by \( wB \).

Let $T$ be the set of diagonal matrices in $G$. Then $T\cong (\C^\ast)^{n}$ and the left multiplication of $T$ on $G$ induces the action of $T$ on $\flag(\C^n)$. The set of $T$-fixed points of $\flag(\C^n)$ consists of coordinate flags, so we can identify $(\flag(\C^n))^T$ with $\S_n$. Furthermore, two $T$-fixed points $v$ and $w$ are connected by a one-dimensional orbit of $T$ if and only if $w=v(i,j)$ and the weight of the orbit closure $\overline{\mathcal{O}_{v,w}}$ is $t_{v(i)}-t_{v(j)}$, where $(i,j)$ is the transposition interchanging $i$ and $j$ for $1\leq i\neq j\leq n$. Therefore, the action of $T$ on $\flag(\C^n)$ has finitely many $T$-fixed points and finitely many one-dimensional $T$-orbits.

The flag variety $\flag(\C^n)$ admits a cell decomposition called the \emph{Bruhat decomposition}:
\begin{equation}\label{eq:Bruhat_Decomp}
    \flag(\C^n)=\bigsqcup_{w\in\S_n} BwB/B = \bigsqcup_{w\in\S_n} B^-wB/B.
\end{equation}
For each $w\in\S_n$, the cells $X_w^\circ\coloneqq BwB/B$ and $\Owo{w}\coloneqq B^-wB/B$ are isomorphic to $\C^{\ell(w)}$ and $\C^{d-\ell(w)}$, respectively, where 
\[
\ell(w)=|\{(i,j)\mid 1\leq i<j\leq n,\,w(i)>w(j)\}|.
\]
Thus, the cohomology of $\flag(\C^n)$ vanishes in odd degrees. Therefore, $\flag(\C^n)$ is a GKM variety associated with the GKM graph $\Gamma=(V,E,\alpha)$, where
\begin{enumerate}
    \item $V=\S_n$;
    \item $E=\{w\to w(i,j)\mid 1\leq i<j\leq n\}$; and
    \item $\alpha(w\to w(i,j))=t_{w(i)}-t_{w(j)}$.
\end{enumerate}

We demonstrate the GKM graph of $\flag(\C^3)$ in Figure~\ref{fig_GKM_n3}.
\begin{figure}
\begin{tikzpicture}[scale = 0.5]
    \foreach \x/\y in {312/30, 321/90, 231/150, 132/-30, 123/-90, 213/-150}{
        \node[shape = circle, fill=none, draw=black, inner sep = 0pt , minimum size=1.2mm] (\x) at (\y:2) {};
    }

    \node[left] at (213) {$213$};
    \node[left] at (231) {$231$};
    \node[below] at (123) {$123$};
    \node[right] at (132) {$132$};
    \node[right] at (312)  {$312$};
    \node[above] at (321) {$321$};

    \draw[line width= 0.4ex, green!50!black,  semitransparent] (123)--(213)
    (312)--(321)
    (231) to (132);

    \draw[line width= 0.4ex, blue,  semitransparent] (123)--(132)
    (231)--(321)
    (213) to (312);

    \draw[line width= 0.4ex, red,  semitransparent] (213)--(231)
    (132)--(312)
    (123) to (321);

        \begin{scope}[xshift=5cm, yshift=-2cm]
    \draw[line width= 0.4ex, green!50!black,  semitransparent] (0,2)--(1,2)
        node[at end, right, black, opacity = 1] {$\pm (t_2 - t_1)$};

    \draw[line width= 0.4ex, blue,  semitransparent] (0,1)--(1,1)
        node[at end, right, black, opacity=1] {$\pm (t_3-t_2)$};

    \draw[line width= 0.4ex, red,  semitransparent] (0,0)--(1,0)
        node[at end, right, black, opacity=1] {$\pm (t_3 - t_1)$};
    \end{scope}
\end{tikzpicture}
\caption{The GKM graph of $\flag(\C^3)$}\label{fig_GKM_n3}
\end{figure}
In~\eqref{eq:Bruhat_Decomp}, the cells $X_w^\circ$ and $\Owo{w}$ are called the \emph{Schubert cell} and the \emph{opposite Schubert cell} indexed by $w$, respectively.  Their closures $X_w\coloneqq \overline{X_w^\circ}$ and $\Ow{w}\coloneqq \overline{\Owo{w}}$ are the \emph{Schubert variety} and the \emph{opposite Schubert variety} indexed by $w$, respectively.

Note that the decompositions in~\eqref{eq:Bruhat_Decomp} are also known as the Bia{\l}ynicki-Birula plus and minus decompositions, respectively, with respect to a suitable generic one-parameter subgroup of $T$. The two cells $X_w^\circ$ and $\Omega_w^\circ$ are the plus and minus Bia{\l}ynicki-Birula cells, respectively.

The inclusion relations among Schubert varieties define a partial order on $\S_n$, called the \emph{Bruhat order}:
\[
v\leq w\qquad \iff \qquad X_v\subseteq X_w.
\]
The minimal and maximal elements of $\S_n$ in the Bruhat order are $e\coloneqq 12\cdots n$ and $w_0\coloneqq n(n-1)\cdots 1$, respectively. Since $B^-=w_0Bw_0$, we get $\Ow{w}=w_0 X_{w_0w}$ and
\[
v\leq w  \qquad \iff \qquad \Ow{v}\supseteq \Ow{w}.
\]
In particular, $X_w$ and $\Ow{w}$ admit cell decompositions:
\begin{equation}\label{eq:decomp_Sch}
    X_w=\bigsqcup_{v\leq w} X_v^\circ\quad\text{ and }\quad \Ow{w}=\bigsqcup_{v\geq w}\Owo{v},
\end{equation} which defines a stratification of each variety.

Note that for each $w\in\S_n$, both $X_w$ and $\Ow{w}$ are $T$-stable and their sets of $T$-fixed points are given by
\[X_w^T=[e,w]\qquad\text{ and }\qquad\Ow{w}^T=[w,w_0],\]
where $[v,w]$ denotes the set of all permutations $z\in\S_n$ satisfying $v\leq z\leq w$ in the Bruhat order.
Therefore, we have the following equivalence:
\[v\leq w\qquad\iff\qquad w\in\Ow{v}^T.\]
From now on, we will only consider the opposite Schubert variety \( \Ow{w} \), so we will omit ``opposite" and refer to \(\Ow{w} \) as the Schubert variety if there is no confusion.

\subsection{Hessenberg varieties}\label{subsec_Hess}
A nondecreasing function $h\colon [n] \to [n]$ is called a \emph{Hessenberg function} if $h(i)\geq i$ for each $i\in [n]$. We often use the list of values $(h(1), h(2), \dots, h(n))$ of $h$ to represent a Hessenberg function $h\colon [n] \to [n]$.

For a linear operator $x$ on $\mathbb{C}^n$, the \emph{Hessenberg variety} determined by $x$ and $h$ is defined as
$$\Hess(x, h)\coloneqq \{V_\bullet \in \flag(\C^n) \,\, |\,\, x(V_i)\subseteq V_{h(i)} \mbox{ for } i=1, \dots, n \} \,.$$
For example, when $h=(n, n, \dots, n)$,  $\Hess(x, h)$ is the flag variety $\flag(\C^n)$.

Note that two Hessenberg varieties $\Hess(x,h)$ and $\Hess(g^{-1}xg,h)$ are isomorphic for $g\in\GL_n(\mathbb{C})$. Hence, if $x$ and $x'$ have the same Jordan form, then $\Hess(x,h)$ and $\Hess(x',h)$ are isomorphic.

Let \( s \) be a regular semisimple linear operator on \( \mathbb{C}^n \); that is, the Jordan form of \( s \) is a diagonal matrix with \( n \) distinct eigenvalues. The Hessenberg variety \( \Hess(s, h) \) determined by \( s \) and a Hessenberg function \( h \) is then called \emph{regular semisimple}, and it is a smooth projective variety of dimension $$d_h\coloneqq \sum_{i=1}^{n} (h(i)-i).$$

We fix a diagonal matrix $s$ with $n$ distinct eigenvalues.
From~\eqref{eq:Bruhat_Decomp}, the regular semisimple Hessenberg variety \( \Hess(s, h) \) is decomposed as follows:
\begin{equation}\label{eq:Hess_paving}
\begin{split}
    \Hess(s, h)
&=  \bigsqcup_{w \in \S_n} \left(\Owo{w} \cap \Hess(s, h)\right).
\end{split}
\end{equation}
In particular, for each $w\in\S_{n}$, the intersection $\Owo{w, h} \coloneqq  \Owo{w} \cap \Hess(s, h)$ is a Białynicki-Birula minus cell and it is isomorphic to the affine space $\C^{d_h-\ell_h(w)}$, where $$\ell_h(w)= \left|  \{ (i,j) \mid i<j,\,\,w(i)> w(j), \,\, j\leq h(i)   \}  \right|\,.$$ Therefore, regular semisimple Hessenberg varieties are equivariantly formal with respect to any algebraic torus action.

We call $\Owo{w, h}$ the \emph{\textup{(}opposite\textup{)} Hessenberg Schubert cell} indexed by $w$, and set $$\Ow{w,h}\coloneqq {\overline{\Owo{w,h}}}$$ which we call the \emph{\textup{(}opposite\textup{)} Hessenberg Schubert variety} indexed by $w$.

\begin{proposition}\cite[Proposition~5.4]{Tymoczko}\label{prop:Tymoczko} A regular semisimple Hessenberg variety $ \Hess(s, h)$  is a $T$-stable subvariety of $\flag(\C^n)$, and it is a GKM variety with the GKM graph $\Gamma=(V, E, \alpha)$, where
\begin{enumerate}
    \item $V= \S_n$,
    \item $E=\{ w\to w(i,j) \mid 1\leq  i<j\leq h(i) \}$, and
    \item $\alpha(w\to w(i,j))=t_{w(i)}-t_{w(j)}$.
\end{enumerate}
\end{proposition}

We present the GKM graphs of regular semisimple Hessenberg varieties for $h = (2,3,3)$ and $h = (2,2,3)$ in Figure~\ref{fig_GKM_233_223}.
\begin{figure}
\begin{tikzpicture}[scale = 0.5]
\begin{scope}[xshift = -8cm]
    \foreach \x/\y in {312/30, 321/90, 231/150, 132/-30, 123/-90, 213/-150}{
        \node[shape = circle, fill=none, draw=black, inner sep = 0pt , minimum size=1.2mm] (\x) at (\y:2) {};
    }

    \node[left] at (213) {$213$};
    \node[left] at (231) {$231$};
    \node[below] at (123) {$123$};
    \node[right] at (132) {$132$};
    \node[right] at (312)  {$312$};
    \node[above] at (321) {$321$};

    \draw[line width= 0.4ex, green!50!black,  semitransparent] (123)--(213)
    (312)--(321);

    \draw[line width= 0.4ex, blue,  semitransparent] (123)--(132)
    (231)--(321);

    \draw[line width= 0.4ex, red,  semitransparent] (213)--(231)
    (132)--(312);

    \node[below of = 123] {$h = (2,3,3)$};
\end{scope}
    \foreach \x/\y in {312/30, 321/90, 231/150, 132/-30, 123/-90, 213/-150}{
        \node[shape = circle, fill=none, draw=black, inner sep = 0pt , minimum size=1.2mm] (\x) at (\y:2) {};
    }

    \node[left] at (213) {$213$};
    \node[left] at (231) {$231$};
    \node[below] at (123) {$123$};
    \node[right] at (132) {$132$};
    \node[right] at (312)  {$312$};
    \node[above] at (321) {$321$};

    \draw[line width= 0.4ex, green!50!black,  semitransparent] (123)--(213);

    \draw[line width= 0.4ex, blue,  semitransparent] (231)--(321);

    \draw[line width= 0.4ex, red,  semitransparent] (132)--(312);

    \node[below of = 123] {$h = (2,2,3)$};

    \begin{scope}[xshift=5cm, yshift=-2cm]
    \draw[line width= 0.4ex, green!50!black,  semitransparent] (0,2)--(1,2)
        node[at end, right, black, opacity = 1] {$\pm (t_2 - t_1)$};

    \draw[line width= 0.4ex, blue,  semitransparent] (0,1)--(1,1)
        node[at end, right, black, opacity=1] {$\pm (t_3-t_2)$};

    \draw[line width= 0.4ex, red,  semitransparent] (0,0)--(1,0)
        node[at end, right, black, opacity=1] {$\pm (t_3 - t_1)$};
    \end{scope}
	\end{tikzpicture}
\caption{The GKM graphs of regular semisimple Hessenberg varieties for $h = (2,3,3)$ and $h = (2,2,3)$}\label{fig_GKM_233_223}
\end{figure}

For each permutation $w\in\S_n$, from~\eqref{eq:decomp_Sch}, the intersection $\Ow{w}\cap\Hess(s,h)$ is $T$-stable and has an affine cell decomposition:
\begin{equation}\label{eq:cell-decomp-Ow}
    \Ow{w}\cap\Hess(s,h)=\bigsqcup_{v\geq w}(\Owo{v}\cap \Hess(s,h)).
\end{equation}
Hence, the intersection $\Ow{w}\cap\Hess(s,h)$ is a GKM variety and its GKM graph is the subgraph of $\Gamma(\Hess(s,h))$ induced by the vertex set $(\Ow{w}\cap\Hess(s,h))^T=[w,w_0]$.
In general, ${\Ow{w}\cap\Hess(s,h)}$ is reducible, and each of its irreducible components is of the form $\Ow{v,h}$ for some $v\in [w,w_0]$. 
Each set on the right-hand side of~\eqref{eq:cell-decomp-Ow} is a Bia{\l}ynicki-Birula cell of $\Hess(s,h)$, which is an affine cell. 

\begin{definition}\label{def_admissible}
    For a Hessenberg function~$h$, a permutation~$w$ is called \emph{$h$-admissible} if it satisfies
    $$w^{-1}(w(j)+1)\leq h(j)\text{ for all $j$ such that }w(j)\leq n-1.$$
\end{definition}
For example, when $h=(3,3,4,4)$, there are twelve $h$-admissible permutations:
\[
1234,1423,2134,2341,2431,3241,3412,3421,4123,4231,4312,4321.
\]

One important property of $h$-admissible permutations is that the associated Hessenberg Schubert cohomology classes generate the cohomology of the Hessenberg variety
$\Hess(s,h)$ as an $\mathfrak S_n$-module. For the definition of the action of $\mathfrak S_n$ on $H^*(\Hess(s,h))$, see \cite{Tymoczko}.

\begin{theorem}[{see \cite[Theorem~A]{CHL2} for details}]\label{thm:h-admissible}
For each $k\geq 0$, the set of Hessenberg Schubert cohomology classes~$[\Omega_{w,h}]$ associated with $h$-admissible permutations satisfying $\ell_h(w)=k$ generates the $\S_n$-module $H^{2k}(\Hess(s,h);\C)$, where $[\Omega_{w,h}]$ is the cohomology class of a Hessenberg Schubert variety $\Omega_{w,h}$ in the cohomology ring  of $\Hess(s,h)$.\footnote{In~\cite{CHL}, $[\Omega_{w,h}]$ is denoted by $\sigma_{w,h}$, which we do not use here.}
\end{theorem}

The main geometric ingredient of the proof of Theorem \ref{thm:h-admissible} is that there is an equivalence relation on $\S_n$ so that the associated Hessenberg Schubert varieties in the same equivalence class share common features, and the set of $h$-admissible permutations serves as a set of representatives of this equivalence relation.
The precise statement is presented below.
\begin{proposition}[cf.~\cite{CHL2,CHL,HP}]\label{prop:prop_h-admissible}
For a permutation $w\in\S_n$, there is a unique $h$-admissible permutation $\widetilde{w} {\geq w}$ satisfying that
\begin{equation}\label{eq:generator}
\widetilde{w}(i)<\widetilde{w}(j) \quad \text{ if and only if } \quad w(i)<w(j)\, \quad \text{ for all } i<j \text{ with } j\leq h(i).
\end{equation} 
In this case, we have
\begin{equation}\label{eq:Owh}
    \Ow{w,h}=u(\overline{\Owo{\widetilde{w}}\cap\Hess(u^{-1}su,h)}),
\end{equation}
where $u=w\widetilde{w}^{-1}$.
\end{proposition}

Note that for each $w\in\S_n$, the unique $h$-admissible permutation~$\widetilde{w}$ satisfying~\eqref{eq:generator} can be obtained from~\cite[Proposition~3.8]{CHL2}, and for such $\widetilde{w}$ and $u$, \eqref{eq:Owh} holds by~\cite[Lemma~4.5]{CHL}. Alternatively, the same result follows from Lemma 2.11 and Proposition~3.14 in~\cite{HP}.

In what follows, we show that $\Ow{w,h}$ and $\Ow{w}$ have the same $T$-fixed point set if and only if $w\in\S_n$ is $h$-admissible.
\begin{proposition}\label{prop:fixed_points_admissible}
    Let $h \colon [n] \to [n]$ be a Hessenberg function. A permutation $w \in \S_n$ is $h$-admissible if and only if $\Owh{w}^T = [w,w_0]$.
\end{proposition}
\begin{proof}
    By \cite[Theorem 4.8]{HP}, if $w$ is $h$-admissible, then $\Ow{w,h}^T=[w,w_0].$  To prove the converse, suppose that $w$ is not  $h$-admissible but $\Owh{w}^T = [w,w_0]$.
    By Proposition~\ref{prop:prop_h-admissible}, there exists a unique $h$-admissible permutation $\widetilde{w} > w$ such that $\Ow{w,h}=u(\overline{\Owo{\widetilde{w}}\cap\Hess(u^{-1}su,h)})$, where $u=w\widetilde{w}^{-1}$.
    Then we have $\Ow{w,h}^T=u(\overline{\Owo{\widetilde{w}}\cap\Hess(u^{-1}su,h)})^T$, that is, $[w,w_0]=u[\widetilde{w},w_0]$ (cf.~\cite[Corollary 4.12]{HP}). However, since $w < \widetilde{w}$, it follows that $[\widetilde{w}, w_0]\subsetneq [w,w_0]$. Therefore, two intervals $[w,w_0]$ and $[\widetilde{w},w_0]$ have different cardinalities, which contradicts $[w,w_0]=u[\widetilde{w},w_0]$.
\end{proof}
It follows from Proposition~\ref{prop:fixed_points_admissible} that if $w$ is not an $h$-admissible permutation, $\Ow{w}\cap \Hess(s,h)$ is reducible because 
\[
(\Ow{w}\cap\Hess(s,h))^T=\Ow{w}^T=[w,w_0] \supsetneq\Ow{w,h}^T
\]
and \(\Ow{w,h}\) is an irreducible component of \(\Ow{w}\cap\Hess(s,h)\).
%\lee{if and only if $w$ is \park{\(h\)-}admissible.} 
However, the $h$-admissibility does not ensure the irreducibility of the intersection $\Ow{w} \cap \Hess(s,h)$ (cf. Examples~\ref{ex:palindromic-not-regular}, \ref{example_34555}, and ~\cite[Example~4.16]{HP}).

Recall that $\Gamma(\Owh{w})$ is the graph whose vertices are $\Owh{w}^T$ and edges are $T$-stable curves in~$\Owh{w}$.
\begin{corollary} \label{cor. graph}
    For a permutation $w\in\S_n$, let $\widetilde{w}$ be the corresponding $h$-admissible permutation. Then the graphs $\Gamma(\Owh{w})$ and $\Gamma(\Owh{\widetilde{w}})$ are isomorphic as unlabeled graphs.
\end{corollary}
\begin{proof}
    Let $u=w\widetilde{w}^{-1}$. Then $\Owh{\widetilde{w}}^T=[\widetilde{w},w_0]$ and $\Owh{w}^T=u[\widetilde{w},w_0]$. Note that two vertices $z$ and $v$ are connected by an edge in $\Gamma(\Ow{w,h})$ if and only if $u^{-1}z$ and $u^{-1}v$ are connected by an edge in $\Gamma(\Ow{\widetilde{w},h})$. Therefore, $u$ is an isomorphism from $\Gamma(\Owh{\widetilde{w}})$ to $\Gamma(\Owh{w})$.\footnote{In graph theory, an isomorphism of graphs $G$ and $H$ is an edge-preserving bijection between the vertex sets of $G$ and $H$.}
\end{proof}

%%% section 4
\section{Smoothness}\label{sec:Irr} 
Let $T$ be an algebraic torus and let $X$ be a smooth GKM variety with respect to a $T$-action on $X$. Let $Y$ be a $T$-stable smooth subvariety of $X$.
In this section, we consider the intersection of $Y$ with the closure of a Bia{\l}ynicki-Birula cell of $X$. 
Also, we will prove Theorem~\ref{thm_main}.

Let $\mathcal{S}$ be the set of $T$-fixed points in $X$, and denote by $\Omega_w^{\circ}$ the Bia{\l}ynicki-Birula (minus) cell of $X$ with respect to a generic one-parameter subgroup of  $T$, that is, we have the Bia{\l}ynicki-Birula decomposition of $X$ as follows:
    \[
    X = \bigsqcup_{w \in \mathcal{S}} \Omega_w^{\circ}.
    \]
Since $X$ is smooth, each cell $\Omega_w^{\circ}$ is an affine cell. We denote by $\Omega_w$ the closure of $\Omega_w^{\circ}$. We note that the Bia{\l}ynicki-Birula decomposition is not necessarily a stratification.\footnote{For more details on the Bia{\l}ynicki-Birula decomposition, see~\cite{BB}.}

Let $Y$ be a $T$-stable subvariety of $X$. 
If $Y$ is smooth, then it is a GKM variety. In general,
the subvariety $Y$ is not necessarily a GKM variety, but we investigate the geometry of 
$Y$ using its one-skeleton.
\begin{definition}
Let $X$ be a smooth GKM variety with respect to a $T$-action on $X$, and let $Y$ be a $T$-stable subvariety of~$X$.
Define a (undirected) graph $\Gamma(Y)$ by
\begin{itemize}
\item the set of vertices of $\Gamma(Y)$ is $Y^T \subseteq X^T$;
\item $\{v,w\}$ is an edge if and only if there exists a one-dimensional $T$-orbit closure in $Y$ whose fixed points are $v$ and $w$.
\end{itemize}
\end{definition}

By definition, $\Gamma(X)$ is the underlying graph of the GKM graph of $X$, and we have $\Gamma(Y) \subseteq \Gamma(X)$. Notice that the graph $\Gamma(Y)$ is not necessarily an induced subgraph of $\Gamma(X)$. 
For a graph $\Gamma$, let $V(\Gamma)$ denote the set of vertices of $\Gamma$ and let $E(\Gamma)$ denote the set of edges of $\Gamma$. 

\begin{example}
Let $X$ be a flag variety $\flag(\C^3)$. Regarding the left multiplication of a maximal torus $T \cong (\Cstar)^3$, the flag variety $X$ is a GKM variety. There are six $T$-fixed points indexed by the elements of the symmetric group $\mathfrak{S}_3$. Recall from Proposition~\ref{prop:Tymoczko}, the GKM graph of~$X$ has the following edges:
\[
\{(w, w(i,j)) \mid w \in \mathfrak{S}_3,~ 1 \leq i < j \leq 3\}.
\]

Consider the permutohedral variety $Y$ of dimension $2$, which is a torus orbit closure in $X$ satisfying $Y^T = X^T$. The set of $T$-stable curves can be expressed by
\[
\{(w, w(i,i+1)) \mid w \in \mathfrak{S}_3, ~i=1,2\}.
\]
Accordingly, $\Gamma(Y) \subsetneq \Gamma(X)$ while $V(\Gamma(Y)) = V(\Gamma(X))$.
\end{example}

For \(v\in Y^T\), we denote by
\[
E(Y,v)
\coloneqq
\{\,C\subseteq Y
\mid C \text{ is a \(T\)-curve and } v\in C\,\}
\]
the set of \(T\)-curves in \(Y\) passing through \(v\).
Thus, \(E(Y,v)\) is naturally identified with the set of edges of
\(\Gamma(Y)\) incident to \(v\).

\begin{lemma}[{\cite[Lemma in Section~2]{Carrell94}}]\label{lem:number_of_edges}
Let \(X\subset\mathbb P^{N}\) be a \(T\)-stable projective subvariety where $T \subset \GL_{N+1}(\C)$. Suppose that $X^T$ is finite. Then, for each fixed point $x\in X^T$, the number of one-dimensional orbit closures containing~$x$ is at least $\dim_\mathbb{C}(X)$.
\end{lemma}

For $w \in \mathcal{S}$, consider the intersection of $Y$ with the affine cell $\Owo{w}$ and the closure of the intersection:
\[
\OwoY{w} \colonequals \Owo{w} \cap Y, \quad
\OwY{w} \colonequals \overline{\OwoY{w}}.
\]
If $Y$ is a smooth $T$-stable subvariety of $X$, then the set $\{\OwoY{w}\}_{w \in \mathcal{S}}$ of intersections forms the Bia{\l}ynicki-Birula decomposition of $Y$ and each element is an affine cell.

% \subsection{Irreducibility}
The intersection $\Ow{w} \cap Y$ is again a $T$-stable subvariety of $X$ which could be reducible and non-equidimensional in general (see Example~\ref{example_reducible_non_equidimentional}).
Considering the graph~$\Gamma(\Ow{w} \cap Y)$, we can determine the smoothness of the intersection $\Ow{w} \cap Y$.

\begin{proposition}\label{prop_regular_implies_same_graph}
Let $Y$ be a smooth $T$-stable subvariety of a GKM variety $X$. Fix $w \in \mathcal{S} = X^T$ such that $\Owo{w} \cap Y \neq \emptyset$. Let $w=v_0,v_1,\dots,v_r=p$ be a path in $\Gamma(\Ow{w} \cap Y)$. If
\[
|E(\Ow{w} \cap Y, v_i)| = \dim(\OwoY{w}) \quad \text{ for }i=0,\dots,r,
\]
then $v_i \in \OwY{w}$ and $E(\OwY{w},v_i) = E(\Ow{w} \cap Y,v_i)$ for $i=0,\dots,r$.

In particular, if $\Gamma(\Ow{w} \cap Y)$ is connected and regular, then its degree is $\dim(\OwoY{w})$ and $\Gamma(\OwY{w}) = \Gamma(\Ow{w} \cap Y)$.
\end{proposition}
\begin{proof}
We note that since $\Owo{w} \cap Y$ is an affine space and is open in $\Ow{w} \cap Y$, the intersection $\Ow{w} \cap Y$ is smooth at every point of $\Owo{w} \cap Y$, and in particular, at \(w\). %; thus, so is at $w$. 
This implies that we have
\begin{equation}\label{eq_dim_and_number_of_edges}
\lvert E(\Ow{w} \cap Y, w) \rvert = \dim (\OwoY{w}).
\end{equation}
Applying Lemma~\ref{lem:number_of_edges} to $\OwY{w}$, we get $\dim (\OwoY{w}) \leq \lvert E(\OwY{w}, w) \rvert$. On the other hand, we have $E(\OwY{w},w) \subseteq E(\Ow{w} \cap Y,w)$. Accordingly, 
\[
\dim (\OwoY{w}) \leq \lvert E(\OwY{w}, w) \rvert \leq \lvert E(\Ow{w} \cap Y,w) \rvert = \dim (\OwoY{w}).
\]
This proves $E(\OwY{w},w) = E(\Ow{w} \cap Y,w)$.

We proceed by induction along the path on the graph $\Gamma(\Ow{w} \cap Y)$. The assertion holds for $v_0 = w$. Suppose that $v_i \in \OwY{w}$. Since $\OwY{w}$ is irreducible, Lemma~\ref{lem:number_of_edges} gives $\lvert E(\OwY{w}, v_i) \rvert \geq \dim (\OwoY{w})$. 
On the other hand, we have $E(\OwY{w},v_i) \subseteq E(\Ow{w} \cap Y, v_i)$. Considering their cardinalities, we obtain 
\begin{equation}\label{eq_OwY_are_the_same_at_vi}
E(\OwY{w},v_i) = E(\Ow{w} \cap Y, v_i).
\end{equation}

The \(T\)-curve joining $v_i$ and $v_{i+1}$ belongs to $E(\Ow{w} \cap Y, v_i)$.
By~\eqref{eq_OwY_are_the_same_at_vi}, %the corresponding $T$-invariant curve is 
it also belongs to \(E(\OwY{w},v_i)\), and hence it is
contained in $\OwY{w}$. Therefore, $v_{i+1} \in \OwY{w}$. This completes the induction and proves the first assertion. 

Finally, suppose that $\Gamma(\Ow{w} \cap Y)$ is regular and connected.  By~\eqref{eq_dim_and_number_of_edges}, its degree is $\dim(\OwoY{w})$. Since the graph $\Gamma(\Ow{w} \cap Y)$ is connected, the first part of the proposition proves $\Gamma(\OwY{w}) = \Gamma(\Ow{w} \cap Y)$.
\end{proof}

% In this Subsection, we prove Theorem~\ref{thm_main}. 
\begin{proof}[Proof of Theorem~\ref{thm_main}]
Let $d = \dim(\OwoY{w})$. 
If \(\Omega_w\cap Y\) is smooth, then at every fixed point the number of incident \(T\)-curves equals its dimension \(\OwoY{w}\), so the graph is \(d\)-regular. This proves (2) $\implies$ (1).

Consider the direction (1) $\implies$ (2).
Suppose that the graph $\Gamma(\Ow{w} \cap Y)$ is $d$-regular. By Proposition~\ref{prop_regular_implies_same_graph}, we have $\Gamma(\Ow{w} \cap Y) = \Gamma(\OwY{w})$.

Since $Y$ is smooth, we get $T_v Y = \TE(Y,v)$ for any $v \in (\Ow{w} \cap Y)^T$.
Moreover, the GKM condition implies that the tangent lines of distinct
$T$-curves through $v$ are linearly independent.  Therefore,
for every $v\in (\Omega_w \cap Y)^T$,
\[
\begin{aligned}
\TE(\Omega_w \cap Y,v)
&\subseteq T_v(\Omega_w \cap Y)\\
&\subseteq T_v\Omega_w\cap T_vY\\
&=\TE(\Omega_w,v)\cap \TE(Y,v)\\
&=\TE(\Omega_w \cap Y,v).
\end{aligned}
\]
Consequently, \(T_v (\Omega_w \cap Y)=\TE(\Omega_w \cap Y,v)\)
and hence
\[
\dim T_v(\Omega_w \cap Y) = d.
\]

By Proposition~\ref{prop_regular_implies_same_graph}, the point $v$ belongs to $\OwY{w}$. Hence 
\[
d \leq \dim_v (\Ow{w} \cap Y) \leq \dim T_v(\Ow{w} \cap Y) = d.
\]
Therefore, $\Ow{w} \cap Y$ is smooth at every $v \in (\Ow{w} \cap Y)^T$. This proves that $\Omega_w \cap Y$ is smooth. 
Indeed, if the singular locus of $\Omega_w \cap Y$ were nonempty, it would be a nonempty
closed $T$-stable subvariety and would therefore contain a
$T$-fixed point.

Since \(\Gamma(\Ow{w}\cap Y)\) is connected, \(\Ow{w}\cap Y\) is connected. As \(\Ow{w}\cap Y\) is smooth, its irreducible components are pairwise disjoint.
Since $\OwY{w}$ is an irreducible component of $\Ow{w} \cap Y$,
we conclude that
$\OwY{w} = \Ow{w} \cap Y$, which proves the theorem. 
\end{proof}

As one can see in the proof of Theorem~\ref{thm_main}, together with Proposition~\ref{prop_regular_implies_same_graph}, the following \emph{local} version holds.
\begin{corollary}\label{cor:local_smooth}
Let $X$,$Y$, and $\Omega_w$ be the same as in Theorem~\ref{thm_main}. 
Let $p \in (\Omega_w \cap Y)^T$. If there exists a path $\gamma$ connecting $w$ and $p$ such that $\left \vert E(\Ow{w} \cap Y, v) \right \vert = \dim \OwoY{w}$ holds for all $v \in V(\gamma)$, then $\Omega_w \cap Y$ is smooth at~$p$.
\end{corollary}

\begin{remark}
As we saw in the proof of Theorem~\ref{thm_main}, 
the implication (1)\(\implies\)(2) remains valid if
$\dim(T_v \Omega_w \cap T_vY) \leq d$  for every $v \in (\Ow{w} \cap Y)^T$.
\end{remark}

\begin{remark}\label{rmk_irreducible}
    In Theorem~\ref{thm_main}, the regularity condition on the graph $\Gamma(\Ow{w} \cap Y)$ is not necessary for the irreducibility of $\Ow{w} \cap Y$. For instance, if $X = Y = G/B$, then $\Ow{w} \cap Y = \Omega_w$ and so is irreducible for any $w \in X^T$. However, $\Omega_w$ could be singular, and the graph corresponding to $\Ow{w} \cap Y = \Omega_w$ is not regular.
\end{remark}

\begin{proposition}\label{lemma_irreducible}
Let $X$ be a smooth GKM variety with respect to a $T$-action on $X$ whose Bia{\l}ynicki-Birula decomposition
$\{\Omega_w^{\circ}\}_{w \in \mathcal{S}}$ forms a stratification.
Let $Y$ be a $T$-stable smooth subvariety of $X$ with $\Owo{w} \cap Y \neq \emptyset$.
Then $\Omega_{w,Y}$ is an irreducible component of $\Omega_w \cap Y$. Furthermore, any irreducible component of $\Ow{w} \cap Y$ is of the form $\OwY{v}$ for some fixed point $v$ in $\mathcal{S}$. 
\end{proposition}
\begin{proof} 
Since the Bia{\l}ynicki-Birula decomposition forms a stratification, we have 
\begin{equation}\label{eq_Aw}
\Omega_w = \bigcup_{v \in \Omega_w^T} \Omega_v^{\circ},
\end{equation}
where $\Omega_w^T \subseteq \mathcal{S}$ is the set of $T$-fixed points of $\Omega_w$. 
By intersecting both sides of~\eqref{eq_Aw} with $Y$, we obtain
\[
\Ow{w} \cap Y = \bigcup_{v \in \Omega_w^T} \OwoY{v}.
\]
Each set on the right-hand side is a Bia{\l}ynicki-Birula cell of $Y$, which is an affine cell since $Y$ is smooth. Therefore, each irreducible component of the intersection $\Ow{w} \cap Y$ should be of the form $\OwY{v}$ for some $v \in \Omega_w^T$.

We notice that each cell $\OwoY{u}$ has a unique $T$-fixed point $u$, which is the minimal element in $\OwY{u}^T$ with respect to the partial order $\leq$ on $\mathcal{S}$.
Since
\[
w \in (\Ow{w} \cap Y)^T = (\Omega_w)^T \cap Y^T \subseteq \Omega_w^T
\]
and $w$ is the minimal element in $\Omega_w^T$, the closure $\OwY{w}$ is an irreducible component of $\Ow{w} \cap Y$. 
\end{proof}

%%% section 4
\section{Hessenberg varieties in type \texorpdfstring{$A$}{A}}\label{sec:Hess}
Recall from Section~\ref{sec:Prelim} that the flag variety $\flag(\mathbb{C}^n)$ is a smooth GKM variety whose Białynicki-Birula decomposition $\{\Owo{w}\}_{w\in\S_n}$ with respect to the natural $T$-action forms a stratification, and a regular semisimple Hessenberg variety $\Hess(s,h)$ is a $T$-stable smooth subvariety of $\flag(\mathbb{C}^n)$. In this section, we apply the results from Section~\ref{sec:Irr} to the intersection $\Ow{w}\cap\Hess(s,h)$ and prove Theorem~\ref{thm_regular_implies_smoothness}. 
A new ingredient is the equivalence between the smoothness and the palindromicity of the Poincar\'e polynomial for the intersection $\Omega_w \cap \Hess(s,h)$ (Theorem \ref{thm:palindromic}). We also include a brief overview of the pattern avoidance criteria for regularity, previously established in~\cite{CHP}. 
Finally, we study the local smoothness of Hessenberg
Schubert varieties at fixed points corresponding to certain
\(h\)-Bruhat covers.

In Theorem~\ref{thm_main_Hess}, if the intersection $\Ow{w}\cap\Hess(s,h)$ is connected and smooth, then it becomes the Hessenberg Schubert variety $\Owh{w}$. One may ask when the intersection $\Ow{w}\cap\Hess(s,h)$ is connected; equivalently, the corresponding graph $\Gamma(\Ow{w}\cap\Hess(s,h))$ is connected.
 The following statement gives a sufficient condition for $\Gwh{w}$ to be connected.

\begin{proposition}\label{prop:connected_graph}
    The graph $\Gwh{w}$ is connected if either
    \begin{enumerate}
        \item $w$ is $h$-admissible, or
        \item $\Hess(s,h)$ is connected.
    \end{enumerate}
\end{proposition}
\begin{proof}
We first consider the case where $w$ is $h$-admissible. It follows from Proposition~2.11 in~\cite{CHL} that there is a path from $u$ to $w_0$ in $\Gamma(\Ow{w}\cap\Hess(s,h))$ for each $u\in [w,w_0]$, so $\Gamma(\Ow{w}\cap\Hess(s,h))$ is connected.

Now assume that $\Hess(s,h)$ is connected; equivalently, $h(i)>i$ for each $i=1,\dots,n-1$. For each $v\in \S_{n}$, we construct a chain
\begin{equation*}
    v<_h vs_{j_1}<_h vs_{j_1}s_{j_2}<_h \cdots<_h vs_{j_1}\cdots s_{j_k}=w_0.
\end{equation*}
by iteratively choosing indices $i$ such that $v_t(i)< v_t(i+1)$, where $v_0=v$ and $v_t=vs_{j_1}\cdots  s_{j_t}$ for $0\leq t\leq k-1$. At each step, we take $j_{t+1}=i$ for the chosen index $i$.
Then $v_t$ and $v_{t+1}=v_ts_{j_{t+1}}$ are connected by an edge in $\Gamma(\Hess(s,h))$ because $h(i)\geq i+1$ for every $1\leq i<n$.
If we take $v\geq w$, then $v_t\geq w$ for every $1\leq t\leq k$, so $\Gwh{w}$ is connected.
\end{proof}

\begin{remark}
    The graph $\Gwh{w}$ can be connected even when $w$ is not $h$-admissible, provided that $\Hess(s,h)$ is connected. For instance, when $h=(2,3,3)$ and $w=213$, $\Gwh{w}$ is connected, but it is not a regular graph. See Figure~\ref{fig_GKM_213_223}.

    \begin{figure}
\begin{tikzpicture}[scale = 0.5]
\begin{scope}[xshift = -8cm]
    \foreach \x/\y in {312/30, 321/90, 231/150, 213/-150}{
        \node[shape = circle, fill=none, draw=black, inner sep = 0pt , minimum size=1.2mm] (\x) at (\y:2) {};
    }

    \node[left] at (213) {$213$};
    \node[left] at (231) {$231$};
    %\node[below] at (123) {$123$};
    %\node[right] at (132) {$132$};
    \node[right] at (312)  {$312$};
    \node[above] at (321) {$321$};

    \draw[line width= 0.4ex, green!50!black,  semitransparent]
    (312)--(321);

    \draw[line width= 0.4ex, blue,  semitransparent]
    (231)--(321);

    \draw[line width= 0.4ex, red,  semitransparent] (213)--(231);

    \end{scope}

    \begin{scope}[xshift=-1cm, yshift=-1cm]
    \draw[line width= 0.4ex, green!50!black,  semitransparent] (0,2)--(1,2)
        node[at end, right, black, opacity = 1] {$\pm (t_2 - t_1)$};

    \draw[line width= 0.4ex, blue,  semitransparent] (0,1)--(1,1)
        node[at end, right, black, opacity=1] {$\pm (t_3-t_2)$};

    \draw[line width= 0.4ex, red,  semitransparent] (0,0)--(1,0)
        node[at end, right, black, opacity=1] {$\pm (t_3 - t_1)$};
    \end{scope}
	\end{tikzpicture}
\caption{$\Gamma(\Ow{w}\cap\Hess(s,h))$ when $w=213$ and $h=(2,3,3)$}\label{fig_GKM_213_223}
\end{figure}
\end{remark}

\begin{remark}
As an immediate consequence of Proposition~\ref{prop:connected_graph}, we observe the following:
\begin{enumerate}
    \item If the graph $\Gamma(\Omega_w\cap\Hess(s,h))$ is disconnected, then $w$ is not $h$-admissible.
    \item If $\Hess(s,h)$ is connected and $w$ is not $h$-admissible, then $\Omega_w\cap\Hess(s,h)$ is reducible and connected, so $\Omega_w\cap\Hess(s,h)$ is not smooth. Accordingly, $\Gwh{w}$ is not regular.
    \item If $w$ is not $h$-admissible and $\Gwh{w}$ is regular, then $\Hess(s,h)$ is disconnected.
\end{enumerate}
\end{remark}

\subsection{Regularity, palindromicity, and smoothness}

We now investigate the structure of $\Gwh{w}$ for an $h$-admissible permutation~$w$.
Note that for each vertex $u$ in $\Gwh{w}$, the size of the set $E(\Ow{w}\cap\Hess(s,h), u)$ is equal to the size of the following set
\[
E_{w,h}(u) \coloneqq \{ (i,j) \mid u(i,j) \geq w \text{~and~} 1\leq i<j\leq h(i) \},
\]
which consists of the transpositions defining the edges meeting at $u$.
We assume that all transpositions $(i,j)$ satisfy $i<j$ for notational convenience. Using the graph $\Gamma(\Hess(s,h))$, we define the notion of \emph{$h$-Bruhat order} on $\S_n$ as the transitive closure of the following relation:
\[
u<_h v \quad \text{ if and only if }\quad v=u(i,j) \text{ for some $(i,  j)$ such that $j\leq h(i)$ and } \ell(v)>\ell(u).
\]

The following lemma will be used repeatedly in the remainder of this section.
\begin{lemma}\cite[Lemma 3.6]{CHP} \label{lem:CHP}
    Let $w$ be an $h$-admissible permutation. For vertices $u$ and $v$ in $\Gwh{w}$ satisfying that $v=u(a,b)$ for some $(a,b) \in E_{w,h}(u)$ and $u<_h v$, the map $\phi_{uv} \colon E_{w,h}(u) \to E_{w,h}(v)$ by $(i,j) \mapsto (\overline{i},\overline{j})$ is injective, where
    $$
    (\overline{i},\overline{j}) \coloneqq \begin{cases}
    (b,j) & \mbox{if $i=a$, $j>b$, and $(b,j)\not \in E_{w,h}(u) $},\\
    (i,a) & \mbox{if $i<a$, $j=b$, and $(i,a)\not \in E_{w,h}(u) $},\\
    (i,j) & \mbox{otherwise}.
    \end{cases}
    $$
\end{lemma}
\begin{example}
For $h = (3,3,4,4)$ and $w = 2134$, we demonstrate Lemma~\ref{lem:CHP}. We refer the reader to Figure~\ref{figure_Graphs_2134} for the graph $\Gamma(\Omega_w \cap \Hess(s,h))$. We first notice that two exceptional cases in Lemma~\ref{lem:CHP} occur only when $i=a < j < b$ or $i < a < j=b$.

Take $u = 3142$ and $v = 3412 = 3142(2,3)$, that is $(a,b) = (2,3)$. In this case, we have
\[
E_{w,h}(u) = \{(1,3), (2,3), (3,4)\} \quad \text{and} \quad
E_{w,h}(v) = \{(1,2), (2,3), (3,4)\}.
\]
For $(i,j) = (2,3)$ or $(3,4)$, neither $i=a < b < j$ nor $i < a < j=b$. Therefore, $\phi_{uv}(2,3) = (2,3)$ and $\phi_{uv}(3,4) = (3,4)$. For $(i,j) = (1,3) \in E_{w,h}(u)$, we have $1 < a=2$, $3=b$, and moreover, $(1,2) \notin E_{w,h}(u)$. Therefore, we get $\phi_{uv}(1,3) = (1,2)$.
\end{example}

Let $w$ be an $h$-admissible permutation. We denote by $\deg_{w,h}(u)$ the number of edges incident to $u$ in $\Gwh{w}$, that is,
\[
\deg_{w,h}(u)=|E_{w,h}(u)|=|E(\Ow{w}\cap \Hess(s,h),u)|.
\]
It follows from Lemma~\ref{lem:CHP} that
$$u<_h v \implies \deg_{w,h}(u)\leq \deg_{w,h}(v).$$

\begin{corollary} \label{cor regularity}
    Let $w$ be an $h$-admissible permutation. For a vertex $u\in [w,w_0]$, if $\deg_{w,h}(u)=\dim \Owo{w,h}$, then $\Ow{w}\cap\Hess(s,h)$ is smooth at each $v$ with $w\leq_h v\leq_h u$. In particular, if $\deg_{w,h}(w_0)=\dim\Owo{w,h}$, then $\Ow{w}\cap\Hess(s,h)=\Ow{w,h}$ and it is smooth.
\end{corollary}
\begin{proof}
    Recall that if $w$ is $h$-admissible, each $u\in [w,w_0]$ satisfies that $w \leq_h  u  \leq_h  w_0$ by \cite[Proposition~2.11]{CHL} and \cite[Lemma~3.2]{CHP}. By Lemma~\ref{lem:CHP}, if $\deg_{w,h}(u)=\dim \Owo{w,h}$, then $\deg_{w,h}(v)=\dim\Owo{w,h}$ for every vertex $v$ with $w\leq_h v\leq_h u$. Therefore, the result follows from Corollary~\ref{cor:local_smooth}.

    If we further assume that $\deg_{w,h}(w_0)=\dim\Owo{w,h}$, then the graph $\Gamma(\Ow{w}\cap\Hess(s,h))$ is regular by Lemma~\ref{lem:CHP}. By applying Theorem~\ref{thm_main_Hess}, the result follows.
\end{proof}

Since \[\Omega_w\cap\Hess(S,h)=\bigsqcup_{v\geq w}\bigl(\Omega_v^\circ\cap\Hess(s,h)\bigr),\qquad \Omega_v^\circ\cap\Hess(s,h)\cong\mathbb{C}^{d_h-\ell_h(v)},\] the Poincar\'e polynomial of \(\Omega_w\cap\Hess(s,h)\) is
\[
P_{w,h}(q)
=
\sum_{v\geq w} q^{d_h-\ell_h(v)},
\]
where \(q\) is assigned cohomological degree~\(2\).

\begin{theorem}\label{thm:palindromic}
Let \(w\in\mathfrak{S}_n\) be \(h\)-admissible. Then the following are equivalent:
\begin{enumerate}
    \item \(\Omega_w\cap \Hess(s,h)\) is smooth.
    \item $\Omega_w\cap\Hess(s,h)$ is irreducible, and \(P_{w,h}(q)\) is palindromic.
    \item \(\Gamma(\Omega_w\cap \Hess(s,h))\) is regular.
\end{enumerate}
\end{theorem}

\begin{proof}
Let us consider the implication \((1)\implies (2)\). Since \(w\) is \(h\)-admissible,  \(\Ow{w}\cap\Hess(s,h)\) is connected by Proposition~\ref{prop:connected_graph}. Its smoothness implies the irreducibility, and the palindromicity of $P_{w,h}(q)$ is clear. On the other hand,
\((3)\implies (1)\) follows from
Theorem~\ref{thm_main_Hess} and Proposition~\ref{prop:connected_graph}.
Hence, it remains to prove \((2)\implies (3)\).

For simplicity, write \(\Gamma\coloneqq \Gamma(\Omega_w\cap \Hess(s,h)).
\)
Then \(
V(\Gamma)=[w,w_0].
\)
Note that
\[
P_{w,h}(1)=|V(\Gamma)|
\]
and
\[
P_{w,h}'(1)
=
\sum_{v\geq w}
\bigl(d_h-\ell_h(v)\bigr).
\]
For each \(v\geq w\), the quantity \(d_h-\ell_h(v)\) is equal to the
number of upward edges of \(\Gamma\) emanating from \(v\). Therefore,
summing over all \(v\geq w\) counts each edge of \(\Gamma\) exactly once,
and hence
\[
P_{w,h}'(1)=|E(\Gamma)|.
\]

Assume that \(\Omega_w\cap\Hess(s,h)\) is irreducible and \(P_{w,h}(q)\) is palindromic. Since 
\[
\deg P_{w,h}(q)=d_h-\ell_h(w),
\]
we have
\[
q^{d_h-\ell_h(w)}P_{w,h}(q^{-1})
=
P_{w,h}(q).
\]
Differentiating both sides and substituting \(q=1\), we obtain
\[
\bigl(d_h-\ell_h(w)\bigr)P_{w,h}(1)
=
2P_{w,h}'(1).
\]
Therefore,
\begin{equation}\label{eq:deg-edge}
\bigl(d_h-\ell_h(w)\bigr)|V(\Gamma)|
=
2|E(\Gamma)|.
\end{equation}

By the handshaking lemma,
\[
\sum_{v\in V(\Gamma)}
\deg_{w,h}(v)
=
2|E(\Gamma)|.
\]
Combining this with~\eqref{eq:deg-edge}, we obtain
\begin{equation}\label{eq:ids}
\sum_{v\in V(\Gamma)}
\deg_{w,h}(v)
=
\sum_{v\in V(\Gamma)}
\bigl(d_h-\ell_h(w)\bigr).
\end{equation}

On the other hand, since $w$ is $h$-admissible, Lemma~\ref{lem:CHP} implies that
\[
\deg_{w,h}(v)
\geq
\deg_{w,h}(w)
=
d_h-\ell_h(w)
\]
for every \(v\geq w\).
Together with~\eqref{eq:ids}, this forces
\[
\deg_{w,h}(v)
=
d_h-\ell_h(w)
\]
for every \(v\geq w\).
Therefore, \(\Gamma\) is regular.
\end{proof}
The following example shows that the irreducibility assumption in Theorem~\ref{thm:palindromic}(2) cannot be omitted.
\begin{example}\label{ex:palindromic-not-regular}
Consider $h=(2,4,5,5,5)$ and $w=12534\in\mathfrak{S}_5$.
Then
\[
d_h=\sum_{i=1}^{5}(h(i)-i)=6.
\]
Moreover, \(w\) is \(h\)-admissible and \(\ell_h(w)=2\). Indeed,
\[
w^{-1}(2)=2\leq h(1),\quad
w^{-1}(3)=4\leq h(2),\quad
w^{-1}(4)=5\leq h(4),\quad
w^{-1}(5)=3\leq h(5).
\]

A direct enumeration of the permutations \(v\geq w\), according to
their \(h\)-lengths, gives
\[
\begin{array}{c|cccccc}
k
&1&2&3&4&5&6\\ \hline
\bigl|\{v\geq w\mid \ell_h(v)=k\}\bigr|
&1&8&27&27&8&1
\end{array}
\]
Therefore, 
\[ 
P_{w,h}(q)=1+8q+27q^2+27q^3+8q^4+q^5,
\]
which is palindromic.

On the other hand, the GKM graph
\(
\Gamma\bigl(\Omega_w\cap\Hess(s,h)\bigr)
\)
is not regular, since
\[
\deg_{w,h}(w)=4
\qquad\text{whereas}\qquad
\deg_{w,h}(w_0)=6.
\]

Note that, for \(v=51234\geq w\),
\[
\ell_h(v)=1<2=\ell_h(w).
\]
This implies that the dimension of the Hessenberg Schubert variety $\Ow{v,h}$ is greater than that of $\Ow{w,h}$. Therefore, the intersection $\Ow{w} \cap \Hess(s,h)$ is reducible at $v$.
% Hence, \(\Omega_{v,h}\) is an irreducible component of
% \(\Omega_w\cap\Hess(S,h)\) distinct from \(\Omega_{w,h}\). 
% In particular,
% \(\Omega_w\cap\Hess(s,h)\) is reducible.
In particular, the cell~$\Owh{w}^{\circ}$ corresponding to \(w\) is not of maximal
dimension in the intersection $\Ow{w} \cap \Hess(s,h)$. This explains why the palindromicity does not imply the regularity
in this example.
\end{example}
Before we prove Theorem~\ref{thm_regular_implies_smoothness}, we recall 
the pattern avoidance criteria for the regularity of $\Gwh{w}$ established in~\cite{CHP}. For further details, we refer the reader to~\cite{CHP}. Let $h$ be a Hessenberg function on $[n]$. For an $h$-admissible permutation $w\in \S_n$,
we say that $w$ contains the associated pattern $\hpat{\pi}$ if there exist indices $i<j<k<\ell$ satisfying the following conditions:
\begin{enumerate}
    \item \textbf{$\hpat{2143}$}: $w(j)<w(i)<w(\ell)<w(k)$, with $i<j<k<\ell\leq h(i)$.

    \item \textbf{$\hpat{1324}$}: $w(i)<w(k)<w(j)<w(\ell)$, with $i<j<k<\ell\leq h(j)$ and $k \leq h(i)$.

    \item \textbf{$\hpat{1243}$}: $w(i)<w(j)<w(\ell)<w(k)$, with $i<j<k<\ell\leq h(j)$ and $j \leq h(i)<\ell$.

    \item \textbf{$\hpat{2134}$}: $w(j)<w(i)<w(k)<w(\ell)$, with $i<j<k<\ell\leq h(k)$ and $k \leq h(i)<\ell$.

    \item \textbf{$\hpat{1423}$}: $w(i)<w(k)<w(\ell)<w(j)$, with $i<j<k<\ell\leq h(j)$ and $k \leq h(i)<\ell$.

    \item \textbf{$\hpat{2314}$}: $w(k)<w(i)<w(j)<w(\ell)$, with $i<j<k<\ell\leq h(j)$ and $k \leq h(i)<\ell$.

    \item \textbf{$\hpat{2413}$}: $w(k)<w(i)<w(\ell)<w(j)$, with $i<j\leq h(i)<k\leq h(j)<\ell\leq h(k)$.
\end{enumerate}

\begin{theorem}\cite[Theorem~C]{CHP}\label{thm:pattern_avoidance}
    For an $h$-admissible permutation $w$, the graph $\Gwh{w}$ is regular if and only if $w$ avoids all  associated patterns
    $$\hpat{2143},\,\hpat{1324},\,\hpat{1243},\,\hpat{2134},\,\hpat{1423},\,\hpat{2314},\,\hpat{2413}.\,$$
\end{theorem}

Let us prove Theorem~\ref{thm_regular_implies_smoothness}.

\begin{proof}[Proof of Theorem~\ref{thm_regular_implies_smoothness}]
The equivalence of statements (1)--(4) follows from  Theorems~\ref{thm:palindromic} and~\ref{thm:pattern_avoidance}. Hence,
it is enough to prove the last statement.
Note that the GKM graph of $\Hess(s,h)$ is independent of the choice of $s$, so $\Gwh{w}$ does not depend on the choice of $s$. Hence, if
the graph $\Gwh{\widetilde{w}}$ is regular, then $\Gamma(\Ow{\widetilde{w}}\cap\Hess(u^{-1}su,h))$ is also regular.  Moreover, this graph is connected  by Proposition~\ref{prop:connected_graph}. 
It follows from Theorem~\ref{thm_main_Hess} that $\Omega_{\widetilde{w}} \cap \Hess(u^{-1}su,h) = \overline{\Ow{\widetilde{w}}^\circ\cap\Hess(u^{-1}su,h)}$ and is smooth.
Since
    \[
        \Owh{w}=u(\overline{\Ow{\widetilde{w}}^\circ\cap\Hess(u^{-1}su,h)})
    \]
    by Proposition~\ref{prop:prop_h-admissible}, the Hessenberg Schubert variety $\Owh{w}$ is smooth.
\end{proof}

\begin{remark}
    When $w$ is not $h$-admissible, the pattern $\hpat{2413}$ is replaced by four patterns: $\hpat{25314}$, $\hpat{24315}$, $\hpat{14325}$, and $\hpat{15324}$. If $w$ avoids all ten patterns, then $\Ow{w,h}$ is smooth; see Definition~4.9 and Theorem~4.14 in \cite{CHP} for details. However, this does not imply that the graph $\Gwh{w}$ is regular. Note that the graph $\Gamma_{w,h}$ in \cite{CHP} is the subgraph of $\Gamma(\Hess(s,h))$ induced by $\Ow{w,h}^T$, so we have $\Gamma(\Ow{w,h})\subseteq \Gamma_{w,h}\subseteq \Gamma(\Ow{w}\cap\Hess(s,h))$, where each inclusion may be proper. Indeed, the second inclusion is an equality if and only if $w$ is $h$-admissible.
\end{remark}

%%%%
\subsection{Smoothness at \(h\)-Bruhat covers}
We recall a well-known criterion for determining whether \( v \leq w \) in the Bruhat order. Given a subset \( S \subseteq [n] \), we write \( S \!\!\uparrow \) for the tuple obtained by listing the elements of \( S \) in increasing order.

\begin{proposition}\cite[Theorem 2.6.3]{BjBr05}\label{prop:Bruhat order}
    For $u, v\in \mathfrak{S}_n$, the following statements are equivalent:
    \begin{enumerate}
        \item $u\leq v$ in the Bruhat order.
        \item $u[k]\!\!\uparrow\leq v[k]\!\!\uparrow$ for all $k\in [n]$.
    \end{enumerate}
    Here, we denote by $u[k]$ the set $\{u(1),u(2),\dots,u(k)\}$.
\end{proposition}

\begin{definition}
Let \(t=(a,b)\) be a reflection with \(a<b\). We say that a Hessenberg
function \(h\) is \emph{transitive with respect to \(t=(a,b)\)} if the following conditions hold:
\begin{enumerate}
    \item \(a<b\leq h(a)\);
    \item for every \(i<a\), if \(a\leq h(i)\), then \(b\leq h(i);
    \)
    \item for every \(j>b\), if \(j\leq h(b)\), then \(j\leq h(a)\); equivalently, \(h(a)=h(b)\).
\end{enumerate}
\end{definition}

We write \(u\lessdot_h v\) if \(v\) covers \(u\) in the \(h\)-Bruhat order; that is,
\(u<_h v\) and there is no \(x\in\mathfrak{S}_n\) such that
\(u<_h x<_h v\). In this case, we say $v$ is an \emph{$h$-Bruhat cover} of $u$.

\begin{proposition}\label{prop:smooth-h-Bruhat-cover}
Let \(w\in\mathfrak{S}_n\), and let \(\widetilde{w}\) be the unique
\(h\)-admissible permutation corresponding to \(w\).
Suppose that \(\widetilde{w}\lessdot_h
v=\widetilde{w}(a,b)\)
and that \(h\) is transitive with respect to the reflection \((a,b)\).
Then \(\Omega_{\widetilde{w}}\cap\Hess(s,h)\)
is smooth at \(v\).

Consequently,
\(\Omega_{w,h}\) is smooth at \(z=w(a,b)\in\Omega_{w,h}^T\)
whenever \(\widetilde{w}\lessdot_h\widetilde{w}(a,b)\)
and \(h\) is transitive with respect to \((a,b)\).
\end{proposition}

\begin{proof}
Set $u=w\widetilde{w}^{-1}$. By  Proposition~\ref{prop:prop_h-admissible}, for $z\in\Omega_{w,h}^T$, if $\Ow{\widetilde{w},h}$ is smooth at $u^{-1}z$, the variety $\Ow{w,h}$ is smooth at $z$.  
Hence it suffices to show that $\Ow{\widetilde{w},h}$ is smooth at $v=\widetilde{w}(a,b)$. 
Moreover, by Corollary~\ref{cor:local_smooth}, it is enough to prove that 
$\deg_{\widetilde{w},h}(v) = \lvert E(\Omega_{\widetilde{w}} \cap \Hess(s,h), v) \rvert = \dim \Owo{w,h}$.

In the following, we will frequently apply Proposition~\ref{prop:Bruhat order}.  We do not mention it explicitly in each case.

For notational simplicity, we now assume that $w$ is $h$-admissible and $v=w(a,b)$. To prove the smoothness of $\Ow{w,h}$ at $v$, by
    Corollary~\ref{cor:local_smooth}, it is enough to show that $\deg_{w,h}(v)=\dim \Owo{w, h}$.
    By Lemma~\ref{lem:CHP}, it suffices to show that the map
    \[
    \phi_{wv}\colon E_{w,h}(w)\to E_{w,h}(v)
    \]
    is surjective.

    For each $(i,j)\in E_{w,h}(v)$, there are three possibilities: $|\{i,j\}\cap\{a,b\}| = 2$, $1$, or $0$.
    Note that $a<b$ and $i<j$. If $\{i,j\}=\{a,b\}$, then $\phi_{wv}(a,b)=(a,b)$. If \(\{i,j\}\cap\{a,b\}=\emptyset\), then
\((i,j)\in E_{w,h}(v)\) implies \(w(i)<w(j)\), since
\(v=w(a,b)\) is an \(h\)-Bruhat cover of \(w\).
Hence, \((i,j)\in E_{w,h}(w)\). Indeed, if we apply Proposition~\ref{prop:Bruhat order} to $w(a,b)(i,j)>w$, then the desired result can be obtained in each case by choosing $k$ as follows:
    \[
    \begin{cases}
        k = i & \text{ if }i<a \text{ or }i>b;\text{ and}\\
        k = b & \text{ if }a<i<b.
    \end{cases}
    \]
    Thus, we
    have only to show that $\phi_{wv}^{-1}(i,j)$ is nonempty for the transposition $(i,j)\in E_{w,h}(v)$ with $|\{i,j\}\cap\{a,b\}|=1$. There are four possible cases to consider.  Throughout these cases, we use the fact that \( w(a) < w(b) \), which follows from \( w < v = w(a,b) \) in the Bruhat order by Proposition~\ref{prop:Bruhat order}.

\begin{description}\setlength{\itemsep}{1.2em}
    \item[Case 1. $i=a<j<b$] The following table lists the one-line notation for each of the permutations $w$, $v$, $v(a,j)$, and $w(a,j)$.
    \[
    \begin{array}{rcccccccc}
        &&& i=a & & j & & b &\\
        w&=&\cdots&w(a)&\cdots&w(j)&\cdots&w(b)&\cdots\\
        v&=&\cdots&w(b)&\cdots&w(j)&\cdots&w(a)&\cdots\\
        v(a,j)&=& \cdots & w(j) & \cdots & w(b) & \cdots & w(a)&\cdots\\
        w(a,j)&=&\cdots&w(j)&\cdots&w(a)&\cdots&w(b)&\cdots
    \end{array}
    \]
    If $(a,j) \in E_{w,h}(v)$, then $v(a,j) > w$, so $w(a)<w(j)$ by Proposition~\ref{prop:Bruhat order}.
    Therefore, $w(a,j)>w$, which implies $(a,j) \in E_{w,h}(w)$. Accordingly, $\phi_{wv}(a,j) = (a,j)$.

    \item[Case 2. $a<i<j=b$] The following table lists the one-line notations for the permutations $w$, $v$, $v(i,b)$, and $w(i,b)$.
    \[
    \begin{array}{rcccccccc}
        &&& a & & i & & j=b &\\
        w&=&\cdots&w(a)&\cdots&w(i)&\cdots&w(b)&\cdots\\
        v&=&\cdots&w(b)&\cdots&w(i)&\cdots&w(a)&\cdots\\
        v(i,b)&=& \cdots & w(b) & \cdots & w(a) & \cdots & w(i)&\cdots\\
        w(i,b)&=&\cdots&w(a)&\cdots&w(b)&\cdots&w(i)&\cdots
    \end{array}
    \]
    If $(i,b) \in E_{w,h}(v)$, then $v(i,b) > w$, which implies that $w(i) < w(b)$.
    Hence $w(i,b)>w$, which implies $(i,b) \in E_{w,h}(w)$, so $\phi_{wv}(i,b) = (i,b)$.

    \item[Case 3. $i < a$ and $j \in \{a, b\}$] First, consider the case when $j=a$. The following table displays the one-line notation for each of the permutations $w$, $v$, $v(i,a)$, and $w(i,b)$.
    \[
    \begin{array}{rcccccccc}
        &&& i & & a & & b &\\
        w&=&\cdots&w(i)&\cdots&w(a)&\cdots&w(b)&\cdots\\
        v&=&\cdots&w(i)&\cdots&w(b)&\cdots&w(a)&\cdots\\
        v(i,a)&=& \cdots & w(b) & \cdots & w(i) & \cdots & w(a)&\cdots\\
        w(i,b)&=&\cdots&w(b)&\cdots&w(a)&\cdots&w(i)&\cdots
    \end{array}
    \]
    If $(i,a)\in E_{w,h}(v)$, then $v(i,a)>w$, so $w(i) < w(b)$. 
    Moreover, since \((i,a)\in E_{w,h}(v)\), we have \(a\leq h(i)\).
Since \(i<a\) and \(h\) is transitive with respect to \((a,b)\),
it follows that \(b\leq h(i).\)
Thus, \((i,b)\) is an \(h\)-allowed transposition. Together with
\(w(i)<w(b)\), we obtain \((i,b)\in E_{w,h}(w)\). If $(i,a) \notin E_{w,h}(w)$, then $\phi_{wv}(i,b) = (i,a)$; otherwise, we have $\phi_{wv}(i,a) = (i,a)$.

    We now turn to the case when $j=b$.
    The one-line notations for $w$, $v$, $v(i,b)$, and $w(i,b)$ are shown below.
    \[
    \begin{array}{rcccccccc}
        &&& i & & a & & b &\\
        w&=&\cdots&w(i)&\cdots&w(a)&\cdots&w(b)&\cdots\\
        v&=&\cdots&w(i)&\cdots&w(b)&\cdots&w(a)&\cdots\\
        v(i,b)&=& \cdots & w(a) & \cdots & {w(b)} & \cdots & {w(i)}&\cdots\\
        w(i,b)&=&\cdots&w(b)&\cdots& {w(a)}&\cdots& {w(i)}&\cdots
    \end{array}
    \]
    If $(i,b)\in E_{w,h}(v)$, then from $v(i,b) > w$, we have $w(i) < w(a)$ and thus $(i,a)   \in E_{w,h}(w)$. Since $w(a) < w(b)$, we have $(i,b) \in E_{w,h}(w)$, so $\phi_{wv}(i,b) = (i,b)$.

    \item[Case 4. $i \in \{a, b\}$ and $j > b$] The following table lists the one-line notation for each of the permutations $w$, $v$, $v(a,j)$, and $w(a,j)$.
    \[
    \begin{array}{rcccccccc}
        &&& a & & b & & j &\\
        w&=&\cdots&w(a)&\cdots&w(b)&\cdots&w(j)&\cdots\\
        v&=&\cdots&w(b)&\cdots&w(a)&\cdots&w(j)&\cdots\\
        v(a,j)&=& \cdots & w(j) & \cdots & w(a) & \cdots & w(b)&\cdots\\
        w(a,j)&=&\cdots&w(j)&\cdots&w(b)&\cdots&w(a)&\cdots
    \end{array}
    \]If $(a,j)\in E_{w,h}(v)$, then
    $w(a)<w(j)$, so $(a,j)\in E_{w,h}(w)$ and $(b,j) \in E_{w,h}(w)$. Accordingly, $\phi_{wv}(a,j)=(a,j)$.

    It remains to consider the case $(b,j)\in E_{w,h}(v)$. Listed below are the one-line notations for $w$, $v$, $v(b,j)$, and $w(a,j)$.
    \[
    \begin{array}{rcccccccc}
        &&& a & & b & & j &\\
        w&=&\cdots&w(a)&\cdots&w(b)&\cdots&w(j)&\cdots\\
        v&=&\cdots&w(b)&\cdots&w(a)&\cdots&w(j)&\cdots\\
        v(b,j)&=& \cdots & w(b) & \cdots & w(j) & \cdots & w(a)&\cdots\\
        w(a,j)&=&\cdots&w(j)&\cdots&w(b)&\cdots&w(a)&\cdots
    \end{array}
    \] Since $v(b,j)> w $,
    we have $w(a)<w(j)$.
    Since \((b,j)\in E_{w,h}(v)\), we also have \(b<j\leq h(b).\)
By the transitivity of \(h\) with respect to \((a,b)\), it follows that
\(j\leq h(a).\) Together with
\(w(a)<w(j)\), we obtain \((a,j)\in E_{w,h}(w).\) If $(b,j)\notin E_{w,h}(w)$, then $\phi_{wv}(a,j)=(b,j)$. Otherwise, $\phi_{wv}(b,j)=(b,j)$.
\end{description}
This proves the proposition.
\end{proof}
%%%%%%%%%

%%%%%
\section{Further questions}\label{sec:questions}
We have been considering the smoothness of the Hessenberg Schubert variety $\Owh{w}$ by studying the intersection $\Ow{w} \cap \Hess(s,h)$, as described in Theorem~\ref{thm_main_Hess}. In particular, we have provided a sufficient condition, formulated in terms of the graph $\Gwh{w}$, under which the intersection $\Omega_w \cap \Hess(s,h)$ is irreducible, and moreover, $\Owh{w}$ is smooth. However, 
the regularity of \(\Gwh{w}\) is not necessary for the irreducibility in the general setting (see Remark~\ref{rmk_irreducible}), nor is it necessary for the smoothness of \(\Omega_{w,h}\) (see Example~\ref{example_reducible}).
In this section, we present additional questions regarding the structure of $\Ow{w} \cap \Hess(s,h)$ that may offer deeper insights into the geometry of the Hessenberg Schubert variety $\Owh{w}$.

Regarding the irreducibility of the intersection $\Ow{w} \cap \Hess(s,h)$, we present the following question:
\begin{question}\label{question_1}
Let $h \colon [n] \to [n]$ be a Hessenberg function and $w \in \S_n$.
\begin{enumerate}
\item \label{question_1.1} Find a necessary and sufficient condition on $w$ such that $\Omega_w \cap \Hess(s,h)$ is irreducible.
\item  \label{question_1.2} For $v \geq w$, find a necessary and sufficient condition on $v$ such that $\Omega_w \cap \Hess(s,h)$ is irreducible at $v$.
\item \label{question_1.3} Find an explicit description of the irreducible components of $\Omega_w \cap \Hess(s,h)$, that is, find $C_w \subseteq \S_n$ such that
\[
\Omega_w \cap \Hess(s,h) = \bigcup_{v \in C_w} \overline{\Owo{v}  \cap \Hess(s,h)},
\]
where each $\overline{\Omega_v^{\circ} \cap \Hess(s,h)}$  for $v \in C_w$ is an irreducible component in $\Omega_w \cap \Hess(s,h)$.
\end{enumerate}
\end{question}

Considering \eqref{question_1.1} and \eqref{question_1.2} of Question~\ref{question_1}, the existence of $u\in [w,w_0]$ distinct from $w$ with $\ell_h(u)\leq \ell_h(w)$ implies that $\Ow{w}\cap\Hess(s,h)$ is reducible. One may ask whether the condition $\ell_h(u)>\ell_h(w)$ for every $u >_h  w$ implies the irreducibility of $\Ow{w}\cap\Hess(s,h)$. The following example shows that this is not true.
\begin{example}\label{example_34555}
    Let $h=(3,4,5,5,5)$ and $w=25314$. Then $w$ is $h$-admissible. Using SageMath~\cite{sagemath}, we can check that $\ell_h(u)>\ell_h(w)$ for every $u$ with $w<u\leq w_0$. However, $\Ow{w}\cap\Hess(s,h)$ is reducible at $v= 52341$.
\end{example}

Related to Question~\ref{question_1}\eqref{question_1.3}, one
may wonder whether the intersection $\Omega_w \cap \Hess(s,h)$ is equidimensional. We saw in Example~\ref{ex:palindromic-not-regular} that the irreducible components might have different dimensions. In the following example, we examine how to find the irreducible components. 
\begin{example}\label{example_reducible_non_equidimentional}
Let $h = (3,3,4,4)$ and $w = 3214$.
Then $w$ is not $h$-admissible and $\widetilde{w}=4312$ is the corresponding $h$-admissible permutation. Using Proposition~\ref{prop:prop_h-admissible}, one can check that $\Ow{w,h}^T=\{3214,3241\}$ and $\Ow{w,h}$ is a curve.
On the other hand, considering the graph $\Gamma(\Omega_w \cap \Hess(s,h))$ as shown in Figure~\ref{figure_3214_3344}, one can find other irreducible components:
\begin{equation}\label{eq_h3344_w3214}
\Omega_{3214} \cap \Hess(s, h)
= \Omega_{3214, h}  \cup \Omega_{3412, h} \cup \Omega_{3241,h} \cup \Omega_{4213, h}.
\end{equation}
For instance, $\Omega_{3412, h}$ is one of the irreducible components because 
\[
\{v \in [w, w_0] \mid 3412 \in \Ow{v,h}^T\} = \{3412\}.
\]
Similarly, we can see that $\Owh{3241}$ and $\Owh{4213}$ are also irreducible components. We notice that the irreducible component $\Omega_{3412,h}$ is of dimension $2$ since $\ell_h(3412) = 2$. Accordingly, the intersection $\Omega_{3214} \cap \Hess(s,h)$ is not equidimensional.

To prove that there are no further irreducible components beyond those listed in~\eqref{eq_h3344_w3214}, we consider the local chart $\mathscr{O}_{4321}$ consisting of matrices of the following form:
\[
\begin{pmatrix}
x_{41} & x_{42} & x_{43} & 1 \\
x_{31} & x_{32} & 1 & 0 \\
x_{21} & 1 & 0 & 0 \\
1 & 0 & 0 & 0
\end{pmatrix}.
\]
Here, we set $\mathscr{O}_v := v U^-$.
Since any irreducible component of the intersection $\Ow{w} \cap \Hess(s,h)$ is of the form $\Owh{u}$ for some permutation $u$, the remaining possible irreducible components are $\Owh{4312}$, $\Owh{4231}$, $\Owh{3421}$, and $\Owh{4321}$ in this case. All of them contain $4321$ as an element, so it is enough to compute the number of irreducible components meeting at $4321$.
As considered in~\cite[\textsection 5]{DeMari_Shayman_88}, the defining ideal for the intersection $\mathscr{O}_{4321} \cap \Omega_{3214} \cap \Hess(s,h) $ is given by
\[
\begin{split}
&\langle
x_{41}, x_{31}, x_{32}x_{41} - x_{31}x_{42}, - x_{42}x_{21} + x_{41}, -x_{42}, \\
&\qquad - x_{43}x_{32}x_{21} + x_{42}x_{21} + 2 x_{43}x_{31} - 3 x_{41}, x_{43}x_{32} - 2 x_{42}\rangle.
\end{split}
\]% (x3, x1, x2*x3 - x1*x4, -x0*x4 + x3, -x4, -x0*x2*x5 + x0*x4 + 2*x1*x5 - 3*x3, x2*x5 - 2*x4) of Multivariate Polynomial Ring in x0, x1,
Here, we take $s = \text{diag}(1,2,3,4)$ in the computation.
Notice that the above ideal is not prime, and its primary decomposition is given by
\[
\langle x_{42}, x_{41}, x_{32}, x_{31} \rangle \cap
\langle x_{43}, x_{42}, x_{41}, x_{31} \rangle.
% [Ideal (x4, x3, x2, x1) of Multivariate Polynomial Ring in x0, x1, x2, x3, x4, x5 over Rational Field,
 % Ideal (x5, x4, x3, x1)
\]
This implies that exactly two irreducible components meet at $4321$. Indeed, the ideals in the above decomposition correspond to $\Owh{3412}$ and $\Owh{3241}$, respectively. This proves~\eqref{eq_h3344_w3214} holds, that is, $C_w = \{3214, 3412, 3241, 4213\}$. Here, we denote by $C_w$ the set of permutations encoding the irreducible components (see~Question~\ref{question_1}\eqref{question_1.3}).
\begin{figure}
\centering
\begin{subfigure}[b]{0.45\textwidth}
\centering
\begin{tikzpicture} %[>=stealth,->,shorten >=2pt,looseness=.5,auto]
\tikzstyle{every node}=[font=\scriptsize, label distance=1mm]
    \matrix [matrix of math nodes,column sep={0.55cm,between origins},
	row sep={1cm,between origins},
		nodes={circle, draw, inner sep = 0pt , minimum size=1.2mm}]
    {
			& & & & & \node[label = {above:4321}] (4321) {} ; & & & & & \\
			& & & \node[label = {above left:4312}] (4312) {} ; & &
			\node[label = {[label distance = 0.1cm]160:4231}] (4231) {} ; & &
			\node[label = {above right:3421}] (3421) {} ; & & & \\
			& \node[label = {above left:4132}] (4132) {} ; & &
			\node[label = {left:4213}] (4213) {} ; & &
			\node[label = {above:3412}] (3412) {} ; & &
			\node[label = {[label distance = 0.1cm]0:2431}] (2431) {} ; & &
			\node[label = {above right:3241}] (3241) {} ; & \\
			\node[label = {left:1432}] (1432) {} ; & &
			\node[label = {left:4123}] (4123) {} ; & &
			\node[label = {[label distance = 0.01cm]180:2413}] (2413) {} ; & &
			\node[label = {[label distance = 0.01cm]0:3142}] (3142) {} ; & &
			\node[label = {right:2341}] (2341) {} ; & &
			\node[label = {right:3214}] (3214) {} ; \\
			& \node[label = {below left:1423}] (1423) {} ; & &
			\node[label = {[label distance = 0.1cm]182:1342}] (1342) {} ; & &
			\node[label = {below:2143}] (2143) {} ; & &
			\node[label = {-30:3124}] (3124) {} ; & &
			\node[label = {below right:2314}] (2314) {} ; & \\
			& & & \node[label = {below left:1243}] (1243) {} ; & &
			\node[label = {180:1324}] (1324) {} ; & &
			\node[label = {below right:2134}] (2134) {} ; & & & \\
			& & & & & \node[label = {below:1234}] (1234) {} ; & & & & & \\
    };

    \begin{scope}[on background layer]
    \draw[line width= 0.6ex, red,  opacity = 0.3 ]
        (1234)--(2134)
        (1243)--(2143)
        (4123)--(4213)
        (3124)--(3214)
        (3412)--(3421)
        (4312)--(4321)

        (1324)--(2314)
        (2413)--(1423) ;

%(7,-5)--(8,-5) node[right, color = black, opaque] {$t_1 - t_2$};

    \draw[line width= 0.6ex, orange, opacity = 0.3]
        (2134)--(2314)
        (2413)--(2431)
        (4213)--(4231)
        (4132)--(4312)
        (1342)--(3142)
        (1324)--(3124)

        (1234)--(3214)
        (3412)--(1432) ;

%(7,-4.5)--(8,-4.5) node[right, color = black, opaque] {$t_1 - t_3$};

    \draw[line width= 0.6ex,  green!50!black, opacity = 0.3]
        (1432)--(4132)
        (1423)--(4123)
        (2143)--(2413)
        (2314)--(2341)
        (3214)--(3241)
        (3142)--(3412)
        (1342) to [bend right, looseness=0.3] (4312)
        (1243)to [bend left, looseness=0.3](4213) ;

%(7,-4)--(8,-4) node[right, color=black, opaque] {$t_1-t_4$};

    \draw[line width= 0.6ex,  blue, opacity = 0.3]
        (1234)--(1324)
        (1423)--(1432)
        (4123)--(4132)
        (4231)--(4321)
        (2341)--(3241)
        (2314)--(3214)

        (2134)--(3124)
        (2431)--(3421)	;

%(7,-3.5)--(8,-3.5) node[right, color=black, opaque] {$t_2-t_3$};

%
    \draw[line width=0.6ex, purple!50!black, opacity = 0.3]
        (1243)--(1423)
        (2413)--(4213)
        (2431)--(4231)
        (3241)--(3421)
        (3142)--(3124)
        (1324)--(1342)	

        (2341)to [bend right, looseness=0.2] (4321)
        (4123)--(2143)	;

%(7,-3)--(8,-3) node[right, color=black, opaque] {$t_2-t_4$};

    \draw[line width=0.6ex, black, opacity = 0.3]
        (1234)--(1243)
        (2134)--(2143)
        (2341)--(2431)
        (3421)--(4321)
        (3412)--(4312)
        (1342)--(1432)

        (3241)--(4231)
        (4132)--(3142) ;

%(7,-2.5)--(8,-2.5) node[right, color=black, opaque] {$t_3-t_4$};	
		
    \draw (3214) node[rectangle, draw, very thick] {};

    \end{scope}
    \foreach \x in {
    (3214), (3241), (3412), (3421), (4213), (4231), (4312), (4321)}{
        \fill \x circle (0.1cm);}

  \foreach \x/\y in {
    (3214)/(3241), (4213)/(4231), (3412)/(4312),
    (3412)/(3421),
    (3241)/(4231), (3241)/(3421), (3421)/(4321),
    (4312)/(4321), (4231)/(4321)
    }{
    \draw[ultra thick] \x -- \y;
  }
		\end{tikzpicture}
\caption{$\Gamma(\Omega_{3214} \cap \Hess(s, (3,3,4,4)))$}
\end{subfigure}
\quad
\begin{subfigure}[b]{0.45\textwidth}
\centering
\begin{tikzpicture} %[>=stealth,->,shorten >=2pt,looseness=.5,auto]
\tikzstyle{every node}=[font=\scriptsize, label distance=1mm]
    \matrix [matrix of math nodes,column sep={0.55cm,between origins},
	row sep={1cm,between origins},
		nodes={circle, draw, inner sep = 0pt , minimum size=1.2mm}]
    {
			& & & & & \node[label = {above:4321}] (4321) {} ; & & & & & \\
			& & & \node[label = {above left:4312}] (4312) {} ; & &
			\node[label = {[label distance = 0.1cm]160:4231}] (4231) {} ; & &
			\node[label = {above right:3421}] (3421) {} ; & & & \\
			& \node[label = {above left:4132}] (4132) {} ; & &
			\node[label = {left:4213}] (4213) {} ; & &
			\node[label = {above:3412}] (3412) {} ; & &
			\node[label = {[label distance = 0.1cm]0:2431}] (2431) {} ; & &
			\node[label = {above right:3241}] (3241) {} ; & \\
			\node[label = {left:1432}] (1432) {} ; & &
			\node[label = {left:4123}] (4123) {} ; & &
			\node[label = {[label distance = 0.01cm]180:2413}] (2413) {} ; & &
			\node[label = {[label distance = 0.01cm]0:3142}] (3142) {} ; & &
			\node[label = {right:2341}] (2341) {} ; & &
			\node[label = {right:3214}] (3214) {} ; \\
			& \node[label = {below left:1423}] (1423) {} ; & &
			\node[label = {[label distance = 0.1cm]182:1342}] (1342) {} ; & &
			\node[label = {below:2143}] (2143) {} ; & &
			\node[label = {-30:3124}] (3124) {} ; & &
			\node[label = {below right:2314}] (2314) {} ; & \\
			& & & \node[label = {below left:1243}] (1243) {} ; & &
			\node[label = {180:1324}] (1324) {} ; & &
			\node[label = {below right:2134}] (2134) {} ; & & & \\
			& & & & & \node[label = {below:1234}] (1234) {} ; & & & & & \\
    };

    \begin{scope}[on background layer]
    \draw[line width= 0.6ex, red,  opacity = 0.3 ]
        (1234)--(2134)
        (1243)--(2143)
        (4123)--(4213)
        (3124)--(3214)
        (3412)--(3421)
        (4312)--(4321)

        (1324)--(2314)
        (2413)--(1423) ;

%(7,-5)--(8,-5) node[right, color = black, opaque] {$t_1 - t_2$};

    \draw[line width= 0.6ex, orange, opacity = 0.3]
        (2134)--(2314)
        (2413)--(2431)
        (4213)--(4231)
        (4132)--(4312)
        (1342)--(3142)
        (1324)--(3124)

        (1234)--(3214)
        (3412)--(1432) ;

%(7,-4.5)--(8,-4.5) node[right, color = black, opaque] {$t_1 - t_3$};

    \draw[line width= 0.6ex,  green!50!black, opacity = 0.3]
        (1432)--(4132)
        (1423)--(4123)
        (2143)--(2413)
        (2314)--(2341)
        (3214)--(3241)
        (3142)--(3412)
        (1342) to [bend right, looseness=0.3] (4312)
        (1243)to [bend left, looseness=0.3](4213) ;

%(7,-4)--(8,-4) node[right, color=black, opaque] {$t_1-t_4$};

    \draw[line width= 0.6ex,  blue, opacity = 0.3]
        (1234)--(1324)
        (1423)--(1432)
        (4123)--(4132)
        (4231)--(4321)
        (2341)--(3241)
        (2314)--(3214)

        (2134)--(3124)
        (2431)--(3421)	;

%(7,-3.5)--(8,-3.5) node[right, color=black, opaque] {$t_2-t_3$};

%
    \draw[line width=0.6ex, purple!50!black, opacity = 0.3]
        (1243)--(1423)
        (2413)--(4213)
        (2431)--(4231)
        (3241)--(3421)
        (3142)--(3124)
        (1324)--(1342)	

        (2341)to [bend right, looseness=0.2] (4321)
        (4123)--(2143)	;

%(7,-3)--(8,-3) node[right, color=black, opaque] {$t_2-t_4$};

    \draw[line width=0.6ex, black, opacity = 0.3]
        (1234)--(1243)
        (2134)--(2143)
        (2341)--(2431)
        (3421)--(4321)
        (3412)--(4312)
        (1342)--(1432)

        (3241)--(4231)
        (4132)--(3142) ;

%(7,-2.5)--(8,-2.5) node[right, color=black, opaque] {$t_3-t_4$};	
		
    \draw (3214) node[rectangle, draw, very thick] {};

    \end{scope}
    \foreach \x in {
    (3214), (3241) }{
        \fill \x circle (0.1cm);}

  \foreach \x/\y in {
    (3214)/(3241)
    }{
    \draw[ultra thick] \x -- \y;
  }
\end{tikzpicture}
\caption{$\Gamma(\Omega_{3214, (3,3,4,4)})$}
\end{subfigure}
\caption{The graphs $\Gamma(\Omega_w \cap \Hess(s,h))$ and $\Gamma(\Owh{w})$ for $h = (3,3,4,4)$ and $w = 3214$}\label{figure_3214_3344}
\end{figure}
\end{example}

Regarding the smoothness of the Hessenberg Schubert variety $\Owh{w}$, we raise the following question:
\begin{question}\label{question_Owh_smooth}
Find a necessary and sufficient condition on $w$ such that $\Omega_{w,h}$ is smooth.
\end{question}

Theorem~\ref{thm_main_Hess} provides one sufficient condition on $w$ such that $\Omega_{w,h}$ is smooth. The following example shows that there is an $h$-admissible permutation $w$ such that $\Ow{w}\cap\Hess(s,h)$ is reducible, but $\Ow{w,h}$ is smooth.

\begin{example}\label{example_reducible}
Since the graph $\Gwh{w}$ is the GKM graph of the intersection $\Ow{w} \cap \Hess(s,h)$, it could be non-regular even when $\Owh{w}$ is smooth. We consider such a case in this example. Let $h = (3,3,4,4)$ and $w = 2134$. For convenience, we set $s
= \text{diag}(1,2,3,4)$ in this computation. Consider the local chart $\mathscr{O}_{2341}$ consisting of matrices of the following form:
\[
\begin{pmatrix}
x_{41} & x_{42} & x_{43} & 1 \\
1 & 0 & 0 & 0 \\
x_{21} & 1 & 0 & 0 \\
x_{31} & x_{32} & 1 & 0
\end{pmatrix}.
\]
Here, we set $\mathscr{O}_v \colonequals v U^-$.
Then, as provided in~\cite[\textsection 5]{DeMari_Shayman_88}, the defining ideal for the intersection $\mathscr{O}_{2341} \cap \Hess(s,h)$ is given by
\[
\langle x_{43}x_{32}x_{21} - x_{42}x_{21} - 2x_{43}x_{31} - x_{41}, -x_{43}x_{32} - 2x_{42} \rangle.
\]
We denote by $\mathcal{I}_{h,w,v}$ the defining ideal for the intersection $\Ow{w}\cap\Hess(s,h) $ in the local chart~$\mathscr{O}_v$.
Moreover, the defining ideal $\mathcal{I}_{h,2134,2341}$ for the intersection $\Ow{2134}\cap\Hess(s,h)$ is given by
\[
\mathcal{I}_{h,2134,2341}=
\langle
x_{41},
x_{43}x_{32}x_{21} - x_{42}x_{21} - 2x_{43}x_{31} - x_{41}, -x_{43}x_{32} - 2x_{42} \rangle.
\]
We notice that the ideal $\mathcal{I}_{h,2134,2341}$ is not a prime ideal, and its primary decomposition is given by
\[
\mathcal{I}_{h,2134,2341} = \langle x_{41}, x_{42}, x_{43} \rangle \cap
\langle x_{41}, x_{43}x_{32} + 2 x_{42}, 3_{42}x_{21} + 2x_{43}x_{31}, 3x_{32}x_{21} - 4 x_{31} \rangle.
\]
To find the defining ideal of $\Owh{2134}$ among two ideals,
consider the intersection $\Ow{2341}\cap\Hess(s,h) $ in the local chart $\mathscr{O}_{2341}$. Then its defining ideal $\mathcal{I}_{h, 2341, 2341}$ is given by
\[
\mathcal{I}_{h,2341,2341} = \langle x_{41}, x_{42}, x_{43} \rangle.
\]
We notice that this ideal is prime, so $\Ow{2341}\cap\Hess(s,h)  = \Owh{2341}$. Moreover, this proves that the defining ideal of $\Owh{2134}$ in the local chart $\mathscr{O}_{2341}$ is $\langle x_{41}, x_{43}x_{32} + 2 x_{42}, 3 x_{42}x_{21} + 2x_{43}x_{31}, 3x_{32}x_{21} - 4 x_{31} \rangle$.

On the other hand, the intersection of $\Owh{2134}$ and $\Owh{2341}$ in $\mathscr{O}_{2341}$ is defined by
\[
\langle x_{41}, x_{42}, x_{43}, 3x_{32}x_{21} - 4x_{31} \rangle,
\]
which defines the permutohedral variety of dimension $2$. Accordingly, the GKM graph of $\Owh{2134}$ is obtained from $\Gamma(\Ow{2134}\cap\Hess(s,h))$ by deleting the edges connecting the following pairs of vertices: $\{2341, 4321\}, \{3241, 4231\}, \{2431, 3421\}$. See Figure~\ref{figure_Graphs_2134}.

\begin{figure}
\centering
\begin{subfigure}[b]{0.45\textwidth}
    		\begin{tikzpicture} %[>=stealth,->,shorten >=2pt,looseness=.5,auto]
		\tikzstyle{every node}=[font=\scriptsize, label distance=1mm]
    \matrix [matrix of math nodes,column sep={0.55cm,between origins},
	row sep={1cm,between origins},
		nodes={circle, draw, inner sep = 0pt , minimum size=1.2mm}]
    {
			& & & & & \node[label = {above:4321}] (4321) {} ; & & & & & \\
			& & & \node[label = {above left:4312}] (4312) {} ; & &
			\node[label = {[label distance = 0.1cm]160:4231}] (4231) {} ; & &
			\node[label = {above right:3421}] (3421) {} ; & & & \\
			& \node[label = {above left:4132}] (4132) {} ; & &
			\node[label = {left:4213}] (4213) {} ; & &
			\node[label = {above:3412}] (3412) {} ; & &
			\node[label = {[label distance = 0.1cm]0:2431}] (2431) {} ; & &
			\node[label = {above right:3241}] (3241) {} ; & \\
			\node[label = {left:1432}] (1432) {} ; & &
			\node[label = {left:4123}] (4123) {} ; & &
			\node[label = {[label distance = 0.01cm]180:2413}] (2413) {} ; & &
			\node[label = {[label distance = 0.01cm]0:3142}] (3142) {} ; & &
			\node[label = {right:2341}] (2341) {} ; & &
			\node[label = {right:3214}] (3214) {} ; \\
			& \node[label = {below left:1423}] (1423) {} ; & &
			\node[label = {[label distance = 0.1cm]182:1342}] (1342) {} ; & &
			\node[label = {below:2143}] (2143) {} ; & &
			\node[label = {-30:3124}] (3124) {} ; & &
			\node[label = {below right:2314}] (2314) {} ; & \\
			& & & \node[label = {below left:1243}] (1243) {} ; & &
			\node[label = {180:1324}] (1324) {} ; & &
			\node[label = {below right:2134}] (2134) {} ; & & & \\
			& & & & & \node[label = {below:1234}] (1234) {} ; & & & & & \\
    };

    \begin{scope}[on background layer]
    \draw[line width= 0.6ex, red,  opacity = 0.3 ]
        (1234)--(2134)
        (1243)--(2143)
        (4123)--(4213)
        (3124)--(3214)
        (3412)--(3421)
        (4312)--(4321)

        (1324)--(2314)
        (2413)--(1423) ;

%(7,-5)--(8,-5) node[right, color = black, opaque] {$t_1 - t_2$};

    \draw[line width= 0.6ex, orange, opacity = 0.3]
        (2134)--(2314)
        (2413)--(2431)
        (4213)--(4231)
        (4132)--(4312)
        (1342)--(3142)
        (1324)--(3124)

        (1234)--(3214)
        (3412)--(1432) ;

%(7,-4.5)--(8,-4.5) node[right, color = black, opaque] {$t_1 - t_3$};

    \draw[line width= 0.6ex,  green!50!black, opacity = 0.3]
        (1432)--(4132)
        (1423)--(4123)
        (2143)--(2413)
        (2314)--(2341)
        (3214)--(3241)
        (3142)--(3412)
        (1342) to [bend right, looseness=0.3] (4312)
        (1243)to [bend left, looseness=0.3](4213) ;

%(7,-4)--(8,-4) node[right, color=black, opaque] {$t_1-t_4$};

    \draw[line width= 0.6ex,  blue, opacity = 0.3]
        (1234)--(1324)
        (1423)--(1432)
        (4123)--(4132)
        (4231)--(4321)
        (2341)--(3241)
        (2314)--(3214)

        (2134)--(3124)
        (2431)--(3421)	;

%(7,-3.5)--(8,-3.5) node[right, color=black, opaque] {$t_2-t_3$};

%
    \draw[line width=0.6ex, purple!50!black, opacity = 0.3]
        (1243)--(1423)
        (2413)--(4213)
        (2431)--(4231)
        (3241)--(3421)
        (3142)--(3124)
        (1324)--(1342)	

        (2341)to [bend right, looseness=0.2] (4321)
        (4123)--(2143)	;

%(7,-3)--(8,-3) node[right, color=black, opaque] {$t_2-t_4$};

    \draw[line width=0.6ex, black, opacity = 0.3]
        (1234)--(1243)
        (2134)--(2143)
        (2341)--(2431)
        (3421)--(4321)
        (3412)--(4312)
        (1342)--(1432)

        (3241)--(4231)
        (4132)--(3142) ;

%(7,-2.5)--(8,-2.5) node[right, color=black, opaque] {$t_3-t_4$};	
		
    \draw (2134) node[rectangle, draw, very thick] {};

    \end{scope}
    \foreach \x in {(2134), (2143), (2314),	(2341),	(2413), (2431), (3124),
        (3142),	(3214),	(3241),	(3412),	(3421),	(4123),	(4132),	(4213),
        (4231),	(4312),	(4321)} {\fill \x circle (0.1cm);}

  \foreach \x/\y in {(2134)/(2314), (2134)/(3124), (2134)/(2143),
    (2143)/(4123), (2143)/(2413), (3124)/(3142), (3124)/(3214),
    (2314)/(2341), (2314)/(3214),
    (4123)/(4132), (4123)/(4213), (2413)/(4213), (2413)/(2431),
    (3142)/(4132), (3142)/(3412), (2341)/(2431), (2341)/(3241),
    (3214)/(3241), (4132)/(4312), (4213)/(4231), (3412)/(4312),
    (3412)/(3421),  (2431)/(3421), (2431)/(4231),
    (3241)/(4231), (3241)/(3421), (3421)/(4321),
    (4312)/(4321), (4231)/(4321)}{
    \draw[ultra thick] \x -- \y;
  }
    \draw[ultra thick]
        (2341)to [bend right, looseness=0.2] (4321);

		\end{tikzpicture}
\caption{$\Gamma(\Omega_{2134} \cap \Hess(s, (3,3,4,4)))$}
\end{subfigure}
\quad
\begin{subfigure}[b]{0.45\textwidth}
    		\begin{tikzpicture}%[>=stealth,->,shorten >=2pt,looseness=.5,auto]
		\tikzstyle{every node}=[font=\scriptsize, label distance=1mm]
    \matrix [matrix of math nodes,column sep={0.55cm,between origins},
	row sep={1cm,between origins},
		nodes={circle, draw, inner sep = 0pt , minimum size=1.2mm}]
    {
			& & & & & \node[label = {above:4321}] (4321) {} ; & & & & & \\
			& & & \node[label = {above left:4312}] (4312) {} ; & &
			\node[label = {[label distance = 0.1cm]160:4231}] (4231) {} ; & &
			\node[label = {above right:3421}] (3421) {} ; & & & \\
			& \node[label = {above left:4132}] (4132) {} ; & &
			\node[label = {left:4213}] (4213) {} ; & &
			\node[label = {above:3412}] (3412) {} ; & &
			\node[label = {[label distance = 0.1cm]0:2431}] (2431) {} ; & &
			\node[label = {above right:3241}] (3241) {} ; & \\
			\node[label = {left:1432}] (1432) {} ; & &
			\node[label = {left:4123}] (4123) {} ; & &
			\node[label = {[label distance = 0.01cm]180:2413}] (2413) {} ; & &
			\node[label = {[label distance = 0.01cm]0:3142}] (3142) {} ; & &
			\node[label = {right:2341}] (2341) {} ; & &
			\node[label = {right:3214}] (3214) {} ; \\
			& \node[label = {below left:1423}] (1423) {} ; & &
			\node[label = {[label distance = 0.1cm]182:1342}] (1342) {} ; & &
			\node[label = {below:2143}] (2143) {} ; & &
			\node[label = {-30:3124}] (3124) {} ; & &
			\node[label = {below right:2314}] (2314) {} ; & \\
			& & & \node[label = {below left:1243}] (1243) {} ; & &
			\node[label = {180:1324}] (1324) {} ; & &
			\node[label = {below right:2134}] (2134) {} ; & & & \\
			& & & & & \node[label = {below:1234}] (1234) {} ; & & & & & \\
    };

    \begin{scope}[on background layer ]
    \draw[line width= 0.6ex, red,  opacity = 0.3 ]
        (1234)--(2134)
        (1243)--(2143)
        (4123)--(4213)
        (3124)--(3214)
        (3412)--(3421)
        (4312)--(4321)

        (1324)--(2314)
        (2413)--(1423) ;

%(7,-5)--(8,-5) node[right, color = black, opaque] {$t_1 - t_2$};

    \draw[line width= 0.6ex, orange, opacity = 0.3]
        (2134)--(2314)
        (2413)--(2431)
        (4213)--(4231)
        (4132)--(4312)
        (1342)--(3142)
        (1324)--(3124)

        (1234)--(3214)
        (3412)--(1432) ;

%(7,-4.5)--(8,-4.5) node[right, color = black, opaque] {$t_1 - t_3$};

    \draw[line width= 0.6ex,  green!50!black, opacity = 0.3]
        (1432)--(4132)
        (1423)--(4123)
        (2143)--(2413)
        (2314)--(2341)
        (3214)--(3241)
        (3142)--(3412)
        (1342) to [bend right, looseness=0.3] (4312)
        (1243)to [bend left, looseness=0.3](4213) ;

%(7,-4)--(8,-4) node[right, color=black, opaque] {$t_1-t_4$};

    \draw[line width= 0.6ex,  blue, opacity = 0.3]
        (1234)--(1324)
        (1423)--(1432)
        (4123)--(4132)
        (4231)--(4321)
        (2341)--(3241)
        (2314)--(3214)

        (2134)--(3124)
        (2431)--(3421)	;

%(7,-3.5)--(8,-3.5) node[right, color=black, opaque] {$t_2-t_3$};

%
    \draw[line width=0.6ex, purple!50!black, opacity = 0.3]
        (1243)--(1423)
        (2413)--(4213)
        (2431)--(4231)
        (3241)--(3421)
        (3142)--(3124)
        (1324)--(1342)	

        (2341)to [bend right, looseness=0.2] (4321)
        (4123)--(2143)	;

    % \draw[line width=0.6ex]
    %     (2341)to [bend right, looseness=0.2] (4321);
%(7,-3)--(8,-3) node[right, color=black, opaque] {$t_2-t_4$};

    \draw[line width=0.6ex, black, opacity = 0.3]
        (1234)--(1243)
        (2134)--(2143)
        (2341)--(2431)
        (3421)--(4321)
        (3412)--(4312)
        (1342)--(1432)

        (3241)--(4231)
        (4132)--(3142) ;

%(7,-2.5)--(8,-2.5) node[right, color=black, opaque] {$t_3-t_4$};	
		
    \draw (2134) node[rectangle, draw, very thick] {};

    \end{scope}
    \foreach \x in {(2134), (2143), (2314),	(2341),	(2413), (2431), (3124),
        (3142),	(3214),	(3241),	(3412),	(3421),	(4123),	(4132),	(4213),
        (4231),	(4312),	(4321)} {\fill \x circle (0.1cm);}

    \foreach \x/\y in {(2134)/(2314), (2134)/(3124), (2134)/(2143),
    (2143)/(4123), (2143)/(2413), (3124)/(3142), (3124)/(3214),
    (2314)/(2341), (2314)/(3214),
    (4123)/(4132), (4123)/(4213), (2413)/(4213), (2413)/(2431),
    (3142)/(4132), (3142)/(3412), (2341)/(2431), (2341)/(3241),
    (3214)/(3241), (4132)/(4312), (4213)/(4231), (3412)/(4312),
    (3412)/(3421),  (2431)/(4231),
     (3241)/(3421), (3421)/(4321),
    (4312)/(4321), (4231)/(4321)}{
    \draw[ultra thick] \x -- \y;
  }
%  \draw (2341)to [bend right, looseness=0.2] (4321);

\end{tikzpicture}
\caption{$\Gamma(\Omega_{2134, (3,3,4,4)})$}
\end{subfigure}
\caption{The graphs $\Gwh{w}$ and $\Gamma(\Owh{w})$ for $h = (3,3,4,4)$ and $w=2134$}\label{figure_Graphs_2134}
\end{figure}
\end{example}


\begin{thebibliography}{10}

\bibitem{BB}
Andrzej~S. Białynicki-Birula.
\newblock Some theorems on actions of algebraic groups.
\newblock {\em Ann. of Math. (2)}, 98:480--497, 1973.

\bibitem{Billey98}
Sara~C. Billey.
\newblock Pattern avoidance and rational smoothness of {S}chubert varieties.
\newblock {\em Adv. Math.}, 139(1):141--156, 1998.

\bibitem{BjBr05}
Anders Bj\"{o}rner and Francesco Brenti.
\newblock {\em Combinatorics of {C}oxeter groups}, volume 231 of {\em Graduate
  Texts in Mathematics}.
\newblock Springer, New York, 2005.

\bibitem{Carrell94}
James~B. Carrell.
\newblock The {B}ruhat graph of a {C}oxeter group, a conjecture of {D}eodhar,
  and rational smoothness of {S}chubert varieties.
\newblock In {\em Algebraic groups and their generalizations: classical methods
  ({U}niversity {P}ark, {PA}, 1991)}, volume 56, Part 1 of {\em Proc. Sympos.
  Pure Math.}, pages 53--61. Amer. Math. Soc., Providence, RI, 1994.

\bibitem{CHL}
Soojin Cho, Jaehyun Hong, and Eunjeong Lee.
\newblock Bases of the equivariant cohomologies of regular semisimple
  {H}essenberg varieties.
\newblock {\em Adv. Math.}, 423:Paper No. 109018, 2023.

\bibitem{CHL2}
Soojin Cho, Jaehyun Hong, and Eunjeong Lee.
\newblock Permutation module decomposition of the second cohomology of a
  regular semisimple {H}essenberg variety.
\newblock {\em Int. Math. Res. Not. IMRN}, 2023(24):22004--22044, 2023.

\bibitem{CHP}
Soojin Cho, JiSun Huh, and Seonjeong Park.
\newblock Towards combinatorial characterization of the smoothness of
  {H}essenberg {S}chubert varieties.
\newblock {\em European Journal of Combinatorics}, 135:104353, 2026.

\bibitem{DeMari_Shayman_88}
Filippo De~Mari and Mark~A. Shayman.
\newblock Generalized {E}ulerian numbers and the topology of the {H}essenberg
  variety of a matrix.
\newblock {\em Acta Appl. Math.}, 12(3):213--235, 1988.

\bibitem{GKM}
Mark Goresky, Robert Kottwitz, and Robert MacPherson.
\newblock Equivariant cohomology, {K}oszul duality, and the localization
  theorem.
\newblock {\em Invent. Math.}, 131(1):25--83, 1998.

\bibitem{GrahamKreiman}
William Graham and Victor Kreiman.
\newblock Tangent spaces and {$T$}-invariant curves of {S}chubert varieties.
\newblock {\em Adv. Math.}, 444:Paper No. 109626, 40, 2024.

\bibitem{HP}
Megumi Harada and Martha Precup.
\newblock {Torus fixed point sets of Hessenberg Schubert varieties in regular
  semisimple Hessenberg varieties}.
\newblock {\em Osaka Journal of Mathematics}, 60(3):637 -- 652, 2023.

\bibitem{LS84}
V.~Lakshmibai and C.~S. Seshadri.
\newblock Singular locus of a {S}chubert variety.
\newblock {\em Bull. Amer. Math. Soc. (N.S.)}, 11(2):363--366, 1984.

\bibitem{sagemath}
{The Sage Developers}.
\newblock {\em {S}age{M}ath, the {S}age {M}athematics {S}oftware {S}ystem},
  2026.

\bibitem{Tymoczko}
Julianna~S. Tymoczko.
\newblock Permutation actions on equivariant cohomology of flag varieties.
\newblock In {\em Toric topology}, volume 460 of {\em Contemp. Math.}, pages
  365--384. Amer. Math. Soc., Providence, RI, 2008.

\end{thebibliography}
\end{document}